\documentclass[reqno]{amsart}
\usepackage{amsthm,amsmath,amssymb}
\usepackage{stmaryrd,mathrsfs}
\usepackage{esint}
\usepackage{tikz}
\usepackage{amsmath}

\usetikzlibrary{patterns,arrows,decorations.pathreplacing}

\allowdisplaybreaks

\swapnumbers \numberwithin{equation}{section}

\theoremstyle{plain}
\newtheorem{theorem}[equation]{Theorem}
\newtheorem{proposition}[equation]{Proposition}
\newtheorem{corollary}[equation]{Corollary}
\newtheorem{lemma}[equation]{Lemma}

\theoremstyle{definition}
\newtheorem{definition}[equation]{Definition}

\theoremstyle{remark}
\newtheorem{remark}[equation]{Remark}

\makeatletter
\@namedef{subjclassname@2010}{%
  \textup{2010} Mathematics Subject Classification}
\makeatother

\begin{document}

\title{Nontrivial examples of $JN_p$ and $VJN_p$ functions}

\author[T. Takala]{Timo Takala}

\address{(T.T.) Aalto University, Department of Mathematics and Systems analysis, Espoo, Finland}
\email{timo.i.takala@aalto.fi}


\date{June 30, 2022}



\maketitle

\begin{abstract}
We study the John-Nirenberg space $JN_p$, which is a generalization of the space of bounded mean oscillation. In this paper we construct new $JN_p$ functions, that increase the understanding of this function space. It is already known that $L^p(Q_0) \subsetneq JN_p(Q_0) \subsetneq L^{p,\infty}(Q_0)$.
We show that if $|f|^{1/p} \in JN_p(Q_0)$, then $|f|^{1/q} \in JN_q(Q_0)$, where $q \geq p$,
but there exists a nonnegative function $f$ such that $f^{1/p} \notin JN_p(Q_0)$ even though $f^{1/q} \in JN_q(Q_0)$, for every $q \in (p,\infty)$.
We present functions in $JN_p(Q_0) \setminus VJN_p(Q_0)$ and in $VJN_p(Q_0) \setminus L^p(Q_0)$, proving the nontriviality of the vanishing subspace $VJN_p$, which is a $JN_p$ space version of $VMO$. We prove the embedding $JN_p(\mathbb{R}^n) \subset L^{p,\infty}(\mathbb{R}^n)/\mathbb{R}$.
Finally we show that we can extend the constructed functions into $\mathbb{R}^n$, such that we get a function in $JN_p(\mathbb{R}^n) \setminus VJN_p(\mathbb{R}^n)$ and another in $CJN_p(\mathbb{R}^n) \setminus L^p(\mathbb{R}^n)/\mathbb{R}$. Here $CJN_p$ is a subspace of $JN_p$ that is inspired by the space $CMO$.
\end{abstract}

\section{Introduction}

In their seminal paper \cite{johnjanirenberg} John and Nirenberg studied the famous space of bounded mean oscillation $BMO$ and proved the profound John-Nirenberg inequality for $BMO$ functions. In the same paper they also defined a generalization of $BMO$ which has since become known as the John-Nirenberg space, or $JN_p$, with parameter $p \in (1,\infty)$. The space $JN_p$ is a generalization of $BMO$ in the sense that the $BMO$ norm of a function is the limit of its $JN_p$ norm when $p$ tends to infinity.

In this paper we define $JN_p$ as in Definition \ref{jnpmaaritelma} below: for a bounded cube $Q_0 \subset \mathbb{R}^n$ and a number $p \in (1,\infty)$, a function $f$ is in $JN_p(Q_0)$ if $f \in L^1(Q_0)$ and
\begin{equation*}
\| f \|_{JN_p(Q_0)}^p
:= \sup \sum_{i=1}^{\infty} |Q_i| \left( \fint_{Q_i} | f - f_{Q_i} | \right)^p < \infty,
\end{equation*}
where the supremum is taken over all countable collections of pairwise disjoint cubes $(Q_i)_{i=1}^{\infty}$ that are contained in $Q_0$. Here $f_{Q_i}$ denotes the integral average of $f$ over $Q_i$:
\begin{equation*}
f_{Q_i}
:= \fint_{Q_i} f
:= \frac{1}{|Q_i|} \int_{Q_i} f.
\end{equation*}

Many other definitions have been used for $JN_p$, some of which are not equivalent with this definition. This is because the space depends on how much we let the sets $Q_i$ overlap. For example many of the definitions in the more general metric measure space, such as in \cite{doublingmeasure,localtoglobal,localtoglobal2,kim}, use balls $B_i$ instead of cubes, and these balls may overlap with each other in some definitions. This results in different spaces.

Other related function spaces include the dyadic $JN_p$ \cite{kim2}, the John-Nirenberg-Campanato spaces \cite{campanato}, their localized versions \cite{localized} and the sparse $JN_p$ \cite{sparse}. For an extensive survey of John-Nirenberg type spaces see \cite{survey}.

It is well-known that $L^p(Q_0) \subset JN_p(Q_0)$. The space $JN_p$ is also embedded in the weak $L^p$ space, $JN_p(Q_0) \subset L^{p,\infty}(Q_0)$. Further both of these inclusions are strict. An example of a function in $JN_p \setminus L^p$ was constructed in a recent paper \cite{nontriviality}. Thus the space $JN_p$ is a nontrivial space between $L^p$ and $L^{p,\infty}$. Other than that, the behaviour of $JN_p$ functions is still very much unknown.

One of the properties with $L^p$ spaces is that $|f|^{1/q} \in L^q$ if and only if $|f|^{1/p} \in L^p$ whenever $p , q \in [1,\infty)$. A similar equivalence holds for weak $L^p$ spaces. Since $JN_p$ is a space between these spaces, this raises the question: does the equivalence hold for $JN_p$ spaces? In Section \ref{pääosio} we give a negative answer to this question by constructing an example of a nonnegative function $f$ such that $f^{1/p} \notin JN_p(Q_0)$ even though $f^{1/q} \in JN_q(Q_0)$ for every $q \in (p,\infty)$.
The previous existing example of a $JN_p$ function does not have this property of distinguishing $JN_p$ spaces with different values of $p$.
We also show that if $|f|^{1/p} \in JN_p$, then $|f|^{1/q} \in JN_q$  for every $q \geq p$.

Additionally we study the vanishing subspace of $JN_p$, which is denoted by $VJN_p$. This subspace has been studied by Brudnyi and Brudnyi \cite{brudnyi} and by Tao et al. \cite{vjnplahde}.
The space $VJN_p$ is defined as a John-Nirenberg space counterpart to the famous space of vanishing mean oscillation, $VMO$, which is a subspace of $BMO$, and was first studied by Sarason \cite{vmolahde}.
It is well-known that $L^p(Q_0) \subset VJN_p(Q_0) \subset JN_p(Q_0)$. Tao et al. proved that $L^p(Q_0) \neq VJN_p(Q_0)$ 
by showing that the double dual space of $VJN_p$ is $JN_p$ 
\cite{vjnplahde}. 
In Section \ref{vjnposio} we present examples of functions in $JN_p \setminus VJN_p$ and in $VJN_p \setminus L^p$, thereby proving that $VJN_p(Q_0) \neq JN_p(Q_0)$.
Both of these functions are based on the same type of fractal construction.

Finally in Section \ref{rn} we study $JN_p$ and $VJN_p$ in $\mathbb{R}^n$ instead of a bounded cube.
The embedding $JN_p \subset L^{p,\infty}$ has been proved in many different ways for a bounded domain - originally in \cite{johnjanirenberg} and further discussion can be found in \cite{doublingmeasure,goodlambda,localtoglobal,localtoglobal2,milman,campanato}.
However, as far as we know, the question of whether this holds for unbounded domains has not been addressed in the literature. For the sake of completeness we prove that this embedding indeed holds in the case of the whole space $\mathbb{R}^n$, if we replace $L^{p,\infty}$ with $L^{p,\infty} / \mathbb{R}$ - the space of functions in $L^{p,\infty}$ modulo constant.

We also consider the space $CJN_p$: a subspace of $JN_p$, that has been studied by Tao et al. \cite{vjnplahde}.
The space $CJN_p$ is defined analogously to the space of continuous mean oscillation, $CMO$, which is a subspace of $BMO$, and was first announced by Neri \cite{cmolahde}.
We study $CJN_p$ only on $\mathbb{R}^n$ as on bounded domains it coincides with $VJN_p$.
It is clear from the definitions that $L^p / \mathbb{R} \subset CJN_p \subset VJN_p \subset JN_p$.
By extending the functions in Section \ref{vjnposio} into the whole space $\mathbb{R}^n$ we get functions that are in $JN_p(\mathbb{R}^n) \setminus VJN_p(\mathbb{R}^n)$ and in $CJN_p(\mathbb{R}^n) \setminus L^p(\mathbb{R}^n) / \mathbb{R}$, proving that the respective inclusions are strict.
This answers \cite[Question 5.8]{vjnplahde} and it partially answers \cite[Question 5.6]{vjnplahde} and \cite[Question 17]{survey}.

Throughout this paper we assume that a function has its domain in the Euclidian space $\mathbb{R}^n$ or in a bounded cube $Q_0 \subset \mathbb{R}^n$. Every cube in $\mathbb{R}^n$ is assumed to have edges parallel to the coordinate axes.

\section{Preliminaries}
\label{preliminaries}

\begin{definition}[$JN_p$]
\label{jnpmaaritelma}
Let $Q_0 \subset \mathbb{R}^n$ be a bounded cube and let $1 \leq p < \infty$. A function $f$ is in $JN_p(Q_0)$ if $f \in L^1(Q_0)$ and there is a constant $K < \infty$ such that
\begin{equation*}
\sum_{i=1}^{\infty} |Q_i| \left( \fint_{Q_i} | f - f_{Q_i} | \right)^p \leq K^p
\end{equation*}
for all countable collections of pairwise disjoint cubes $(Q_i)_{i=1}^{\infty}$ in $Q_0$.
We denote the smallest such number $K$ by $\|f\|_{JN_p}$.
\end{definition}

\begin{remark}
\label{infversio}
$JN_p$ can also be defined with infimums instead of integral averages.
This definition is equivalent with Definition \ref{jnpmaaritelma}, because
\begin{equation*}
\inf_{c_i} \fint_{Q_i} | f - c_i |
\leq \fint_{Q_i} | f - f_{Q_i} |
\leq 2 \inf_{c_i} \fint_{Q_i} | f - c_i |.
\end{equation*}
The infimum approach makes calculations sometimes much more simple.
\end{remark}

One can also define $JN_p$ by using medians which allows us to not use integrals at all. This approach is studied by Myyryläinen \cite{kim}. 
We do not consider $JN_p$ with $p = 1$, because clearly
$JN_1(Q_0) = L^1(Q_0)$, see equation (\ref{hölder}) below. Thus we assume from now on that $p > 1$.

It is clear that if a function $f$ is in $L^p(Q_0)$, then it is also in $JN_p(Q_0)$. This follows from Hölder's inequality:
\begin{equation}
\label{hölder}
\sum_{i=1}^{\infty} |Q_i| \left( \fint_{Q_i} | f - f_{Q_i} | \right)^p
\leq \sum_{i=1}^{\infty} |Q_i| \fint_{Q_i} | f - f_{Q_i} |^p
\leq \sum_{i=1}^{\infty} 2^p \int_{Q_i} |f|^p
\end{equation}
for every partition $(Q_i)_{i=1}^{\infty}$ of $Q_0$ into disjoint cubes.
The inclusion $L^p \subset JN_p$ is strict, but functions in $JN_p \setminus L^p$ are very complicated. 
An example of a function in $JN_p \setminus L^p$ was constructed 
in a recent paper \cite{nontriviality}.

If a function is in $JN_p(Q_0)$, then it is also in the weak $L^p$-space $L^{p, \infty} (Q_0)$. This result is a $JN_p$ space counterpart of the famous John-Nirenberg lemma for $BMO$ functions.
\begin{theorem}
\label{embedding}
Let $1 < p < \infty$, $Q_0 \subset \mathbb{R}^n$ a bounded cube and $f \in JN_p(Q_0)$. Then $f \in L^{p,\infty}(Q_0)$ and
\begin{equation*}
\| f - f_{Q_0} \|_{L^{p,\infty}(Q_0)}
\leq c \| f \|_{JN_p(Q_0)}
\end{equation*}
with some constant $c = c(n,p)$.
\end{theorem}
Conveniently the constant $c$ does not depend on $Q_0$ even though the proof does rely on the boundedness of $Q_0$.
This embedding was originally proved in \cite[Lemma 3]{johnjanirenberg}.
In \cite{localtoglobal} 
this was generalized from cubes into John domains.
In \cite{doublingmeasure} 
and \cite{localtoglobal2} 
the result was proven in a more general metric space with a doubling measure.
In \cite{campanato} 
a similar result was proven for John-Nirenberg-Campanato spaces, which are a generalization of John-Nirenberg spaces.
Milman gave a new characterization of the weak $L^p$ space as the Garsia-Rodemich space and proved the embedding of $JN_p(Q_0)$ in this space \cite{milman}.
Berkovits et al. proved a good-$\lambda$ inequality and used this to prove multiple embedding theorems including the embedding of $JN_p$ to the weak $L^p$-space \cite{goodlambda}.

The inclusion $JN_p(Q_0) \subset L^{p,\infty}(Q_0)$ is strict as well. Consider the counterexample $f(x) = x^{-1/p}$ in the one-dimensional case $n=1$ with $Q_0 = (0,1)$. In this case $f$ is in the weak $L^p$ space.
If we divide $Q_0$ into subintervals $Q_i = \left[ 2^{-i},2^{-i+1} \right) $, it turns out that
\begin{equation*}
|Q_i| \left( \fint_{Q_i} | f - f_{Q_i} | \right)^p
= |Q_1| \left( \fint_{Q_1} | f - f_{Q_1} | \right)^p
> 0
\end{equation*}
for all $i \geq 1$.
Thus it is easy to see that $f \notin JN_p(Q_0)$.
This counterexample can also be easily extended into the multidimensional case $Q_0 \subset \mathbb{R}^n$ by using the following proposition.

\begin{proposition}
\label{laajennus}
Let $Q_0 \subset \mathbb{R}^n$ be a cube, $f \in L^1(Q_0)$, and $\tilde{f}(x,t) := f(x)$ its trivial extension to
$(x,t) \in \tilde{Q}_0 := Q_0 \times [0, l(Q_0)) \subset \mathbb{R}^{n+1}$. Then $f \in JN_p(Q_0)$ if and only if $\tilde{f} \in JN_p(\tilde{Q}_0)$, and
\begin{equation*}
2^{-1/p} \| f \|_{JN_p(Q_0)} l(Q_0)^{1/p}
\leq \big\| \tilde{f} \big\|_{JN_p(\tilde{Q}_0)}
\leq \| f \|_{JN_p(Q_0)} l(Q_0)^{1/p}.
\end{equation*}
\end{proposition}

For the proof we refer to \cite[Proposition 4.1]{nontriviality}.
In conclusion we know that
\begin{equation*}
L^p(Q_0)
\subsetneq JN_p(Q_0)
\subsetneq L^{p,\infty} (Q_0).
\end{equation*}
This raises many questions about the properties of $JN_p$ functions. For example it is clear that for nonnegative $L^p$ functions the equivalence
\begin{equation}
\label{lpyhtälö}
f^{1/p} \in L^p(Q_0)
\iff
f^{1/q} \in L^q(Q_0)
\end{equation}
holds for all $p,q \in [1, \infty)$. This also holds for weak $L^p$ spaces.
\begin{equation*}
f^{1/p} \in L^{p,\infty}(Q_0)
\iff
f^{1/q} \in L^{q,\infty}(Q_0).
\end{equation*}
Our example in Section \ref{pääosio} proves that $JN_p$ spaces do not have the same property.
However there is a weaker result.

\begin{proposition}
\label{toimiikunpienempi}
If $f: Q_0 \rightarrow \mathbb{R}$ is a nonnegative function, then
\begin{equation*}
f^{1/p} \in JN_p(Q_0)
\text{  implies that  }
f^{1/q} \in JN_q(Q_0)
\end{equation*}
whenever $1<p \leq q<\infty$.
\end{proposition}

\begin{proof}
Firstly we notice that $f^{1/q} \in L^1(Q_0)$, because $f^{1/q} \leq f^{1/p}$ whenever $f \geq 1$, $f^{1/p} \in L^1(Q_0)$ by assumption, and $Q_0$ is a bounded cube with finite measure.
We notice that for any cube $Q_i \subset Q_0$ and for any $c_i \in \mathbb{R}$
\begin{align*}
\left( \inf_{c_i' \in \mathbb{R}} \fint_{Q_i} \left| f^{1/q} - c_i' \right| \right)^q
&\leq \left( \fint_{Q_i} \left| f^{1/q} - |c_i|^{p/q} \right| \right)^q
\\
&\leq \left( \fint_{Q_i} \left| f^{1/q} - |c_i|^{p/q} \right|^{q/p} \right)^p
\\
&\leq \left( \fint_{Q_i} \left| f^{1/p} - c_i \right| \right)^p.
\end{align*}
Here we first used Hölder's inequality and then the inequality $a^r + b^r \leq (a+b)^r$, whenever $a,b \geq 0$ and $r \geq 1$. Also $| f^{1/p} - |c_i| | \leq | f^{1/p} - c_i |$. By taking the infimum over $c_i$, this implies that
\begin{equation*}
|Q_i| \left( \inf_{c_i' \in \mathbb{R}} \fint_{Q_i} \left| f^{1/q} - c_i' \right| \right)^q
\leq 
|Q_i| \left( \inf_{c_i \in \mathbb{R}} \fint_{Q_i} \left| f^{1/p} - c_i \right| \right)^p
\end{equation*}
and consequently
\begin{align*}
\left\| f^{1/q} \right\|_{JN_q(Q_0)}^q
&\leq
\sup_{Q_i \subset Q_0}
\sum_{i=1}^{\infty}
|Q_i| \left( 2 \inf_{c_i' \in \mathbb{R}} \fint_{Q_i} \left| f^{1/q} - c_i' \right| \right)^q
\\
&\leq
2^q \sup_{Q_i \subset Q_0}
\sum_{i=1}^{\infty}
|Q_i| \left( \inf_{c_i \in \mathbb{R}} \fint_{Q_i} \left| f^{1/p} - c_i \right| \right)^p
\leq 2^q \left\| f^{1/p} \right\|_{JN_p(Q_0)}^p.
\end{align*}
Here we used Remark \ref{infversio}. This completes the proof.
\end{proof}

The implication in Proposition \ref{toimiikunpienempi} does not hold in the other direction.
In the next section we construct a function for which
\begin{equation}
\label{esimerkintavoite}
f^{1/q} \in JN_q(Q_0)
\text{  and  }
f^{1/p} \notin JN_p(Q_0)
\end{equation}
where $1 < p < q < \infty$.

\section{An example that distinguishes $JN_p$ spaces with different parameters $p$}
\label{pääosio}

In this section we construct a function $f$ that distinguishes $JN_p$ spaces with different $p$ in the sense that $f^{1/p} \notin JN_p$ but $f^{1/q} \in JN_q$, where $q > p$.
Clearly we need to have $f^{1/q} \in JN_q \setminus L^q$, because if $f^{1/q} \in L^q$, then it follows from (\ref{lpyhtälö}) that $f^{1/p} \in L^p \subset JN_p$.
This is why the example is influenced by the function $g \in JN_p \setminus L^p$ given in \cite{nontriviality}.
Since the spaces $L^p$ and $JN_p$ coincide for monotone functions \cite[Theorem 2.1]{nontriviality}, all examples in $JN_p \setminus L^p$ must have a highly oscillatory structure.
It is enough to construct our function in the interval $[0,1]$. Then we can make a change of variable to get a similar function in an arbitrary interval.

\begin{theorem}
\label{päätulos}
Let $Q_0 = [0,1] \subset \mathbb{R}$ and let $1 < p < \infty$. Then there exists a nonnegative function $f$ such that $f^{1/q} \notin JN_q(Q_0)$ whenever $1 < q \leq p$, and $f^{1/q} \in JN_q(Q_0)$ for every $q > p$.
\end{theorem}

Let us denote the function $f^{1/p}$ in Theorem \ref{päätulos} by $u$ and the function in \cite{nontriviality} by $g$. Both $u$ and $g$ consist of “towers” of width $l_i$ where $i \in \mathbb{Z}_+$. Each tower has an inclined roof with its left side set at height $a_i$ and the right side at height $b_i$. The roof of the tower is linear. On both sides of the tower, at distance $d_i$, there are two similar narrower towers of width $l_{i+1}$ and height ranging from $a_{i+1}$ to $b_{i+1}$. This construction continues indefinitely, with every tower at a given level $i$ having two towers of the next level $i+1$ on both side at distance $d_i$. So there is one tower at level 1, two towers at level 2, 4 towers at level 3 etc. See Figure 1 to get a better qualitative understanding of the function.

To give the exact definition let us use the following dyadic notation. For any dyadic interval $I \subset \left[ 0, \frac{1}{2} \right) $ of length $2^{-i}$, there is a corresponding interval $\hat{I} \subset Q_0$ of length $l_i$. In \cite{nontriviality} the indexing is started from $[0,1)$, but here we start it from $ \left[ 0, \frac{1}{2} \right) $ for technical reasons. It doesn't change the fact that $g \in JN_p \setminus L^p$.
The locations of $\hat{I}$ in $Q_0$ are defined recursively. For $I^1 := \left[ 0, \frac{1}{2} \right) $, let $\hat{I^1}$ be the interval of length $l_1$ located in the center of $Q_0$. If $\hat{I}$ is already defined and $I'$ is the left (right) half of $I$, let $\hat{I'}$ be the interval of length $l_{i+1}$ positioned on the left (right) side of $\hat{I}$ in such a way that dist$(\hat{I'}, \hat{I}) = d_i$.

We denote the intervals by $\hat{I} = [t_I,t_I+l_i]$. For every interval $\hat{I}$ we define the function
\begin{equation*}
u^I(x) := \frac{b_i - a_i}{l_i} (x - t_I) + a_i
\end{equation*}
for every $x \in \hat{I}$. The function $u^I$ is identically zero elsewhere. The function $g^I$ is defined the same way but with different parameters. Note that $u^I$ and $g^I$ are supported in $\hat{I}$, not in $I$.
Finally we define the function $u$ ($g$) as the sum of all these functions $u^I$ ($g^I$):
\begin{equation*}
u(x) := \sum_{\substack{I \subseteq [0,\frac{1}{2}) \\ I \text{  dyadic  } }} u^I(x)
\end{equation*}
for every $x \in Q_0$.


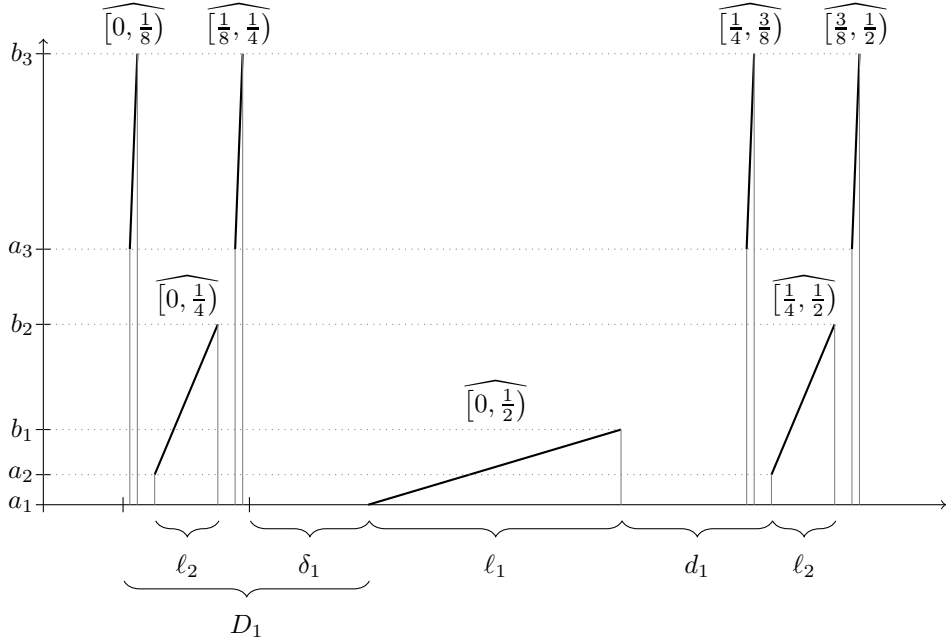
\begin{figure}[ht]
\centering
\begin{tikzpicture}
\draw[black,->] (-6.1,0) -- (6,0);
\draw[black,->] (-6,-0.1) -- (-6,6.2) ;

\draw[black, thick] (1.68,1)--(-1.68,0);
\draw[gray] (1.68,0)--(1.68,1);

\draw (-6.1, 1)--(-5.9,1);
\draw (-6,1) node[left] {$b_1$};
\draw (-6,0) node[left] {$a_1$};
\draw[dotted,gray] (-6,1)--(1.68,1);


\draw[black, thick] (3.68,0.4)--(4.52,2.4);
\draw[gray] (3.68,0)--(3.68,0.4);
\draw[gray] (4.52,0)--(4.52,2.4);

\draw[black, thick] (-3.68,2.4)--(-4.52,0.4);
\draw[gray] (-3.68,0)--(-3.68,2.4);
\draw[gray] (-4.52,0)--(-4.52,0.4);

\draw (-6.1, 2.4)--(-5.9,2.4);
\draw (-6.1, 0.4)--(-5.9,0.4);
\draw (-6,2.4) node[left] {$b_2$};
\draw[dotted,gray] (-6,2.4)--(4.52,2.4);
\draw (-6,0.4) node[left] {$a_2$};
\draw[dotted,gray] (-6,0.4)--(3.68,0.4);

\draw[black, thick](4.75,3.4)--(4.85,6) node[black,midway,yshift=1.7cm] {$\widehat{ \left[ \frac{3}{8},\frac{1}{2} \right) }$};
\draw[gray] (4.75,0)--(4.75,3.4);
\draw[gray] (4.85,0)--(4.85,6);

\draw[black, thick](-4.75,6)--(-4.85,3.4)  node[black,midway,yshift=1.7cm] {$\widehat{ \left[ 0,\frac{1}{8} \right) }$}  ;
\draw[gray] (-4.75,0)--(-4.75,6);
\draw[gray] (-4.85,0)--(-4.85,3.4);

\draw[black, thick](3.45,6)--(3.35,3.4) node[black,midway,yshift=1.7cm] {$\widehat{ \left[ \frac{1}{4},\frac{3}{8} \right) }$};
\draw[gray] (3.45,0)--(3.45,6);
\draw[gray] (3.35,0)--(3.35,3.4);

\draw[black, thick](-3.45,3.4)--(-3.35,6)  node[black,midway,yshift=1.7cm] {$\widehat{ \left[ \frac{1}{8},\frac{1}{4} \right) }$};
\draw[gray] (-3.45,0)--(-3.45,3.4);
\draw[gray] (-3.35,0)--(-3.35,6);

\draw (-6.1, 3.4)--(-5.9,3.4);
\draw (-6.1, 6)--(-5.9,6);
\draw (-6,3.4) node[left] {$a_3$};
\draw (-6,6) node[left] {$b_3$};
\draw[dotted,gray] (-6,3.4)--(4.75,3.4);
\draw[dotted,gray] (-6,6)--(4.85,6);

\draw [yshift=-0.2cm,decorate,decoration={brace,amplitude=6pt,mirror},xshift=0.4pt,yshift=-0.4pt](-1.68,0) -- (1.68,0) node[black,midway,yshift=-0.6cm] { $\ell_1$} node[black,midway,yshift=1.6cm] {$\widehat{ \left[ 0,\frac{1}{2} \right) }$};

\draw [yshift=-0.2cm,decorate,decoration={brace,amplitude=6pt,mirror},xshift=0.4pt,yshift=-0.4pt](3.68,0) -- (4.52,0) node[black,midway,yshift=-0.6cm] { $\ell_2$}  node[black,midway,yshift=3cm] {$\widehat{ \left[ \frac{1}{4},\frac{1}{2} \right) }$} ;

\draw [yshift=-0.2cm,decorate,decoration={brace,amplitude=6pt,mirror},xshift=0.4pt,yshift=-0.4pt](-4.52,0) --(-3.68,0) node[black,midway,yshift=-0.6cm] { $\ell_2$}  node[black,midway,yshift=3cm] {$\widehat{ \left[ 0,\frac{1}{4} \right) }$} ;

\draw [yshift=-0.2cm,decorate,decoration={brace,amplitude=6pt,mirror},xshift=0.4pt,yshift=-0.4pt](-3.26,0) --(-1.68,0) node[black,midway,yshift=-0.6cm] { $\delta_1$}   ;
\draw [black]  (-3.26,-0.1)--(-3.26,0.1);

\draw [yshift=-1cm,decorate,decoration={brace,amplitude=6pt,mirror},xshift=0.4pt,yshift=-0.4pt](-4.94,0) --(-1.68,0) node[black,midway,yshift=-0.6cm] { $D_1$}   ;
\draw [black]  (-4.94,-0.1)--(-4.94,0.1);

\draw [yshift=-0.2cm,decorate,decoration={brace,amplitude=6pt,mirror},xshift=0.4pt,yshift=-0.4pt](1.68,0) -- (3.68,0) node[black,midway,yshift=-0.6cm] { $d_1$};

\end{tikzpicture}

\caption{First three generations $(i = 1, 2, 3)$ in the construction of $u$ (with $p = 2$). $\delta_1$ (resp. $D_1$) is the distance from $\widehat{ \left[ 0,\frac{1}{2} \right) }$ to the nearest (resp. farthest) $\hat{I}$ after infinitely many iterations of the construction. Notice that $a_1 = 0$ by definition.}

\end{figure}

We can recover the function $g$ from \cite{nontriviality} by choosing $a_i = b_i$. Thus $g$ is constant in every interval $\hat{I}$. In this case we denote the height by $h_i$. For $g$ the heights, widths and distances are defined as
\begin{equation}
\begin{cases}
l_i &= 2^{- \left( i+\frac{1}{2} \right)^2} , \\
d_i &= 2^{-(i+1)^2} \text{  and  } \\
h_i &= 2^{i^2/p}
\end{cases}
\label{gparametrit}
\end{equation}
for every $i \geq 1$.
It was already shown that $g \in JN_p$. Then it is easy to see that also $g^{p/q} \in JN_q$ for any $q>1$, because $g^p$ does not depend on $p$.
Therefore it is impossible for that function to separate $JN_p$ spaces with different parameters $p$. The function $g$ had to be modified to get the function $u$ which does prove the result (\ref{esimerkintavoite}).

For the function $u$ the widths, distances and heights are defined as
\begin{equation*}
\begin{cases}
l_i &= 2^{- \left( i+\frac{1}{2} \right)^2} , \\
d_i &= 2^{-(i+1)^2} , \\
a_i &= \left( 1-i^{-1/p} \right) 2^{i^2/p} \text{  and  } \\
b_i &= \left( 1+i^{-1/p} \right) 2^{i^2/p}
\end{cases}
\end{equation*}
for every $i \geq 1$. We will use these definitions from now on. Note that the definitions for $l_i$ and $d_i$ are the same as the respective definitions for $g$. The key difference is that $a_i$ and $b_i$ are different.
It was shown in \cite{nontriviality}, that the intervals $\hat{I}$ are disjoint. Indeed given an interval of length $l_i$, its distance to any other interval $\hat{I}'$ is at least
\begin{equation*}
\delta_i
:= d_i - d_{i+1} - l_{i+2} - d_{i+2} - l_{i+3} - ...
\geq \frac{1}{2} d_i.
\end{equation*}
This means that the functions $u^I$ have disjoint supports. 
Similarly we have a bound for $D_i$, the largest distance from an interval $\hat{I}$ at level $i$ to any of its descendants:
\begin{equation*}
D_i
:= d_i + l_{i+1} + d_{i+1} + l_{i+2} + d_{i+2} + l_{i+3} + ...
\leq 3 d_i.
\end{equation*}
For the proof we refer to \cite[Lemma 3.3]{nontriviality}.
From this it follows that the support of $u$ is contained within an interval of length $l_1 + 2 D_1 \leq \frac{1}{4} + 6 \cdot \frac{1}{16} \leq 1$. This means that the function is indeed supported in $[0,1]$.
It is also useful to notice that $b_i = (1+ i^{-1/p}) 2^{i^2/p} \leq 2 \cdot 2^{i^2/p}$ for every $i \geq 1$.
We summarize the results so far into the following lemma.
\begin{lemma}
For every $i \geq 1$ we have
\begin{align*}
\frac{1}{2} d_i \leq \delta_i \leq d_i &\leq D_i \leq 3 d_i \text{,  and  }
\\
b_i^p &\leq 2^p \cdot 2^{i^2} .
\end{align*}
\end{lemma}

Now that we have defined the function $u = f^{1/p}$, we show that $u \notin JN_p(Q_0)$, which together with Proposition \ref{toimiikunpienempi} implies that $f^{1/q} \notin JN_q(Q_0)$ whenever $1 < q \leq p$.

\begin{proposition}
\label{räjähdys}
$u = f^{1/p} \notin JN_p(Q_0)$.
\end{proposition}

\begin{proof}
Because $u$ is linear in the intervals $\hat{I}$, it has a constant slope within $\hat{I}$, which we denote by $k_I$. Then it is easy to see that
\begin{equation*}
\int_{\hat{I}} \left| u - u_{\hat{I}} \right|
= \frac{|\hat{I}|}{2} \cdot k_I \cdot \frac{|\hat{I}|}{2}
= \frac{l_i^2}{4} \cdot \frac{b_i - a_i}{l_i}.
\end{equation*}
By plugging in the values of $a_i$ and $b_i$ we get
\begin{equation*}
\fint_{\hat{I}} \left| u - u_{\hat{I}} \right|
= \frac{1}{4} \cdot 2 \cdot i^{-1/p} \cdot 2^{i^2/p}.
\end{equation*}
This finally implies that
\begin{align*}
\| u \|_{JN_p(Q_0)}^p
&\geq \sum_{I \text{  dyadic  }} |\hat{I}| \left( \fint_{\hat{I}} \left| u - u_{\hat{I}} \right| \right)^p
\\
&= \sum_{i=1}^{\infty} 2^{i-1} \cdot 2^{-i^2 - i - \frac{1}{4}} \left( \frac{1}{2} \cdot i^{-1/p} \cdot 2^{i^2/p} \right)^p
\\
&= \sum_{i=1}^{\infty} 2^{-1 - \frac{1}{4} -p} \cdot \frac{1}{i}
= \infty,
\end{align*}
since the harmonic series diverges.
\end{proof}

From now on we fix a number $q \in (p,\infty)$. Our goal is to prove that $u^{p/q} = f^{1/q} \in JN_q(Q_0)$. To simplify the notation we denote $u^{p/q}$ by $v$. First we need to prove the following lemmas.

\begin{lemma}
\label{ulkopuolisenyläraja}
Let $I$ be a dyadic interval of length $2^{-i}$.
Then
\begin{equation*}
\int_{\hat{I}} v
\leq c 2^{-i^2/q' - i},
\end{equation*}
where $q' = \frac{q}{q-1}$ is the Hölder conjugate of $q$. Also
\begin{equation*}
\sum_{\substack{I' \subsetneq I \\ I' \text{  dyadic  }}} \int_{\hat{I'}} v
\leq c 2^{-(i+1)^2/q' - i}.
\end{equation*}
In particular $v \in L^1(Q_0)$.
Here the constants $c$ depend only on $p$ and $q$.
\end{lemma}
\begin{proof}
Because $u \leq b_i$ in $\hat{I}$, we have
\begin{equation*}
\int_{\hat{I}} v
\leq | \hat{I} | \cdot b_i^{p/q}
\leq l_i \left( c 2^{i^2} \right)^{1/q}
= 2^{-i^2 - i} \cdot c \cdot 2^{i^2/q}
= c 2^{-i^2/q' -i}.
\end{equation*}
Additionally
\begin{align*}
\sum_{\substack{I' \subsetneq I \\ I' \text{  dyadic  }}} \int_{\hat{I'}} v
&\leq \sum_{j=1}^{\infty} 2^j \cdot c 2^{-(i+j)^2/q' - (i+j)}
= c \sum_{j=1}^{\infty} 2^{-(i+j)^2/q' -i}
\\
&\leq c 2^{(-i^2 - 2i)/q' - i} \sum_{j=1}^{\infty} 2^{- j^2/q'}
\\
&\leq c 2^{-(i+1)^2/q' - i}.
\end{align*}
The fact that $v \in L^1$ follows from the previous result by having $I = \left[ 0,\frac{1}{2} \right)$, because then
\begin{equation*}
\|v\|_{L^1(Q_0)}
= \int_{\hat{I}} v + \sum_{\substack{I' \subsetneq I \\ I' \text{  dyadic  }}} \int_{\hat{I'}} v
\leq c 2^{-1/q' - 1} + c 2^{-2^2/q' - 2}
< \infty.
\end{equation*}
\end{proof}

The following lemma is easy to prove by finding the local minima of the function $h(x) = 2x - \left( 1 + x \right)^{p/q} + \left( 1 - x \right)^{p/q}$, where $0 \leq x \leq 1$, or by using the Taylor approximation $(1+x)^r \approx 1 + rx$, when $|x|$ is small.
\begin{lemma}
\label{taylor}
Let $1 < p < q < \infty$ and $i \in \mathbb{Z}_+$. Then
\begin{equation*}
\left( 1 + i^{-1/p} \right)^{p/q} - \left( 1 - i^{-1/p} \right)^{p/q}
\leq 2 i^{-1/p}.
\end{equation*}
\end{lemma}

To prove that $v \in JN_q(Q_0)$, we let $\mathcal{J}$ be a countable collection of disjoint subintervals of $Q_0$. For $J \in \mathcal{J}$ we define
\begin{equation*}
F(J)
:= |J| \left( \fint_J \left| v - v_J \right| \right)^q.
\end{equation*}
It is clear that if $J$ does not intersect any of the intervals $\hat{I}$, then $F(J) = 0$. Also if $J$ does intersect some of the intervals, then there is a unique widest interval $I_J$, such that $J$ intersects $\hat{I}_J$. We consider four separate cases how $J$ relates to $I_J$. Three of these cases are the same as in \cite{nontriviality}. We added the fourth “contained” one, because this time the function is not constant within the intervals $\hat{I}$.

\begin{definition}
\label{containedshortmediumlong}
\begin{align*}
&1) \text{  } J \text{  is contained if  } |J \setminus \hat{I}_J| = 0 \text{  i.e.  } J \subset \hat{I}_J
\\
&2) \text{  } J \text{  is short if  } 0 < |J \setminus \hat{I}_J| \leq \delta_I
\\
&3) \text{  } J \text{  is medium if  } \delta_I < |J \setminus \hat{I}_J| \leq 2 D_I
\\
&4) \text{  } J \text{  is long if  } |J \setminus \hat{I}_J| > 2 D_I.
\end{align*}
\end{definition}

\begin{lemma}
\label{uusifjraja}
Assume that $J$ intersects some of the intervals $\hat{I}$ and $I = I_J$. Then
\begin{equation*}
F(J)
\leq c \left[ |J|^{1-q} \left( \int_{J \setminus \hat{I} } v \right)^q + \min \big( |J \setminus \hat{I}| , |J \cap \hat{I}| \big) b_I^p + |J \cap \hat{I} | i^{-q/p} \cdot 2^{i^2} \right],
\end{equation*}
where $c = c(q)$ and $|I| = 2^{-i}$.
\end{lemma}

\begin{proof}
It is clear that
\begin{equation*}
v_J
= \frac{1}{|J|} \int_{J \cap \hat{I}} v + \frac{1}{|J|} \int_{J \setminus \hat{I}} v
\end{equation*}
and so we get from the triangle inequality
\begin{align*}
F(J)
&= |J|^{1-q} \left( \int_{J \cap \hat{I}} \left| v - v_J \right| + \int_{J \setminus \hat{I}} \left| v - v_J \right| \right)^q
\\
&\leq |J|^{1-q} \Bigg( \int_{J \cap \hat{I} } \left| v - \frac{1}{|J|} \int_{J \cap \hat{I}} v \right| + \frac{|J \cap \hat{I}|}{|J|} \int_{J \setminus \hat{I}} v
\\
&+ \int_{J \setminus \hat{I}} v + \frac{|J \setminus \hat{I}|}{|J|} \int_{J \cap \hat{I}} v + \frac{|J \setminus \hat{I}|}{|J|} \int_{J \setminus \hat{I}} v \Bigg)^q.
\end{align*}
Define $ \alpha := \frac{1}{|J|} \int_{J \cap \hat{I}} v$. If $ \alpha \leq v(x)$ for all $x \in J \cap \hat{I}$, then $F(J)$ is bounded by
\begin{align*}
&|J|^{1-q} \left( \int_{J \cap \hat{I}} v - \frac{|J \cap \hat{I}|}{|J|} \int_{J \cap \hat{I}} v + 2 \int_{J \setminus \hat{I}} v + \frac{|J \setminus \hat{I}|}{|J|} \int_{J \cap \hat{I}} v \right)^q
\\
= &2^q |J|^{1-q} \left( \int_{J \setminus \hat{I}} v + \frac{|J \setminus \hat{I}|}{|J|} \int_{J \cap \hat{I}} v \right)^q
\\
\leq &c \left[ |J|^{1-q} \left( \int_{J \setminus \hat{I}} v \right)^q + |J|^{1-q} \frac{|J \setminus \hat{I}|^q}{|J|^q} |J \cap \hat{I}|^q \left( b_I^{p/q} \right)^q \right]
\\
\leq &c \left[ |J|^{1-q} \left( \int_{J \setminus \hat{I}} v \right)^q + \min \big( |J \setminus \hat{I}|,|J \cap \hat{I}| \big) b_I^p \right].
\end{align*}
On the other hand if
$\alpha \geq \inf\{ v(x) : x \in J \cap \hat{I} \}$, then we know that $a_i^{p/q} \leq \alpha \leq b_i^{p/q}$ and $a_i^{p/q} \leq v \leq b_i^{p/q}$ in the interval $J \cap \hat{I}$ and so
\begin{align*}
\int_{J \cap \hat{I}} \left| v - \alpha \right|
&\leq \int_{J \cap \hat{I}} b_I^{p/q} - a_I^{p/q}
= | J \cap \hat{I} | \left( b_I^{p/q} - a_I^{p/q} \right)
\\
&= | J \cap \hat{I} | \left( \left( 1 + i^{-1/p} \right)^{p/q} - \left( 1 - i^{-1/p} \right)^{p/q} \right) \left( 2^{i^2/p} \right)^{p/q}
\\
&\leq | J \cap \hat{I} | \cdot 2 \cdot i^{-1/p} \cdot 2^{i^2/q}.
\end{align*}
Here we used Lemma \ref{taylor} to get the final inequality. Then we get
\begin{align*}
F(J)
&\leq |J|^{1-q} \left( | J \cap \hat{I} | \cdot 2 \cdot i^{-1/p} \cdot 2^{i^2/q} + 2 \int_{J \setminus \hat{I}} v + \frac{|J \setminus \hat{I}|}{|J|} \int_{J \cap \hat{I}} v \right)^q
\\
&\leq c |J|^{1-q} |J \cap \hat{I}|^{q} \cdot i^{-q/p} \cdot 2^{i^2} + c |J|^{1-q} \left( \int_{J \setminus \hat{I}} v \right)^q
\\
&+ c |J|^{1-q} \frac{|J \setminus \hat{I}|^q}{|J|^q} |J \cap \hat{I}|^q \left( b_I^{p/q} \right)^q
\\
&\leq c \left[|J|^{1-q} \left( \int_{J \setminus \hat{I}} v \right)^q + \min \big( |J \setminus \hat{I}|,|J \cap \hat{I}| \big) b_I^p + | J \cap \hat{I} | i^{-q/p} \cdot 2^{i^2} \right].
\end{align*}
Hence in all cases the lemma is true.
\end{proof}

\begin{proposition}
\label{containprop}
If $J$ is contained, then
\begin{equation*}
F(J)
\leq c(q) |J| \cdot i^{-q/p} \cdot 2^{i^2}
\end{equation*}
and consequently
\begin{equation*}
\sum_{\substack{J \in \mathcal{J} \\ J \text{  contained  }}} F(J)
\leq c_1(p,q)
< \infty .
\end{equation*}
\end{proposition}
\begin{proof}
In this case the first two terms in Lemma \ref{uusifjraja} are just zero, and so we get the bound for $F(J)$.
Then it follows that
\begin{align*}
\sum_{\substack{J \in \mathcal{J} \\ J \text{  contained  }}} F(J)
&= \sum_{I \text{  dyadic  }} \sum_{\substack{J \in \mathcal{J} \\ J \subset \hat{I} }} F(J)
\leq \sum_{i=1}^{\infty} \sum_{ \substack{I \text{  dyadic  } \\ |I| = 2^{-i}}} \sum_{\substack{J \in \mathcal{J} \\ J \subset \hat{I} }} c |J| \cdot i^{-q/p} \cdot 2^{i^2}
\\
&\leq \sum_{i=1}^{\infty} \sum_{ \substack{I \text{  dyadic  } \\ |I| = 2^{-i}}} c \cdot l_i \cdot i^{-q/p} \cdot 2^{i^2}
\\
&= c \sum_{i=1}^{\infty} 2^{i-1} \cdot 2^{-i^2-i} \cdot i^{-q/p} \cdot 2^{i^2}
\\
&= c \sum_{i=1}^{\infty} i^{-q/p}
= c_1(p,q)
< \infty.
\end{align*}
This is an over-harmonic series, so it converges.
\end{proof}

In Proposition \ref{containprop} above we had a convergence specifically because $q > p$. If we had $p = q$, the corresponding series would diverge as was seen in Proposition \ref{räjähdys}.

We move on to the other cases where $|J \setminus \hat{I}_J| > 0$. In these cases $J$ must intersect the boundary of the interval $\hat{I}_J$ and so for each $I$ there are at most two intervals $J$ such that $I = I_J$ and $|J \setminus \hat{I}| > 0$.

\begin{proposition}
\label{shortprop}
If $J$ is short, then
\begin{equation*}
F(J)
\leq c(p,q) \left( 2^{-2i} + 2^{-i} \cdot i^{-q/p} \right)
\end{equation*}
and consequently
\begin{equation*}
\sum_{\substack{J \in \mathcal{J} \\ J \text{  short  } }} F(J)
\leq c_2(p,q)
< \infty .
\end{equation*}
\end{proposition}
\begin{proof}
In this case $v$ is zero in $J \setminus \hat{I}$ and so the first term in Lemma \ref{uusifjraja} is just zero. Then we get
\begin{align*}
F(J)
&\leq c \left[ \min \big( |J \setminus \hat{I}|,|J \cap \hat{I}| \big) b_I^p + | J \cap \hat{I}| \cdot i^{-q/p} \cdot 2^{i^2} \right]
\\
&\leq c \left[ \delta_I b_I^p + l_I \cdot i^{-q/p} \cdot 2^{i^2} \right]
\\
&\leq c \left( 2^{-i^2 - 2i - 1} \cdot 2^p \cdot 2^{i^2} + 2^{-i^2 - i - \frac{1}{4}} \cdot i^{-q/p} \cdot 2^{i^2} \right)
\\
&\leq c(p,q) \left( 2^{-2i} + 2^{- i} \cdot i^{-q/p} \right).
\end{align*}
Using this bound we get
\begin{align*}
\sum_{\substack{J \in \mathcal{J} \\ J \text{  short  } }} F(J)
&\leq \sum_{i=1}^{\infty} 2^{i-1} \cdot 2 \cdot c \left( 2^{-2i} + 2^{- i} \cdot i^{-q/p} \right)
\\
&= \sum_{i=1}^{\infty} c 2^{-i} + \sum_{i=1}^{\infty} c i^{-q/p}
= c_2(p,q) < \infty.
\end{align*}
\end{proof}

\begin{proposition}
\label{mediumprop}
If $J$ is medium, then
\begin{equation*}
F(J)
\leq c(p,q) \left( 2^{-iq} + 2^{-2i} + 2^{-i} \cdot i^{-q/p} \right)
\end{equation*}
and consequently
\begin{equation*}
\sum_{\substack{J \in \mathcal{J} \\ J \text{  medium  } }} F(J)
\leq c_3(p,q)
< \infty .
\end{equation*}
\end{proposition}
\begin{proof}
In this case the bounds for the second and third term in Lemma \ref{uusifjraja} are essentially the same as in Proposition \ref{shortprop} because $|J \setminus \hat{I}| \leq 2 D_I \leq 6 d_I$. For the first term we use Lemma \ref{ulkopuolisenyläraja} to get
\begin{equation*}
\int_{J \setminus \hat{I}} v
\leq \sum_{\substack{I' \subsetneq I \\ I' \text{  dyadic  }}} \int_{\hat{I'}} v
\leq c 2^{-(i+1)^2/q' - i}.
\end{equation*}
Then we use the fact that $|J| \geq |J \setminus \hat{I}| > \delta_I$ to get
\begin{align*}
|J|^{1-q} \left( \int_{J \setminus \hat{I} } v \right)^q
&\leq \delta_I^{1-q} \left( c 2^{-(i+1)^2/q' - i} \right)^q
\\
&\leq c \left( 2^{-i^2 - 2i - 1} \right)^{1-q} \cdot 2^{-(i+1)^2(q-1)} \cdot 2^{-iq}
\\
&\leq c 2^{-iq}
\end{align*}
and so we have the bound for $F(J)$.
This implies that
\begin{align*}
\sum_{\substack{J \in \mathcal{J} \\ J \text{  medium  } }} F(J)
&\leq \sum_{i=1}^{\infty} 2^{i-1} \cdot 2 c \left( 2^{-iq} + 2^{-2i} + 2^{-i} \cdot i^{-q/p} \right)
\\
&= \sum_{i=1}^{\infty} c 2^{(1-q)i} + \sum_{i=1}^{\infty} c 2^{-i} + \sum_{i=1}^{\infty} c i^{-q/p}
\\
&= c_3(p,q) < \infty.
\end{align*}
\end{proof}

\begin{proposition}
\label{carlesonprop}
If $J$ is long, then
\begin{equation*}
F(J)
\leq c(p,q) |I_J|
\end{equation*}
and the corresponding intervals $I$ form a Carleson family, in the sense that
\begin{equation*}
\sum_{\substack{ I \text{  dyadic  } \\ \exists J \text{  long s.t.  } I = I_J }} |I|
\leq 1.
\end{equation*}
Therefore
\begin{equation*}
\sum_{\substack{J \in \mathcal{J} \\ J \text{  long  } }} F(J)
\leq c_4(p,q)
< \infty.
\end{equation*}
\end{proposition}

\begin{proof}
The proof of the corresponding intervals forming a Carleson family is essentially the same as the proof of \cite[Lemma 3.9]{nontriviality}.

To get the bound for $F(J)$ we use Lemma \ref{uusifjraja}. First we notice that the bound for the first term is essentially the same as in Proposition \ref{mediumprop}, because $|J| \geq |J \setminus \hat{I}| > 2 D_I \geq 2 d_I$. For the second and third term we use the fact that $|J \cap \hat{I}| \leq |\hat{I}| = l_I$. Then we get
\begin{align*}
F(J)
&\leq c \left[ c 2^{-iq} + l_I b_I^p + l_I \cdot i^{-q/p} \cdot 2^{i^2} \right]
\\
&\leq c \left[ 2^{-iq} + 2^{-i^2 - i - \frac{1}{4}} \cdot c 2^{i^2} + 2^{-i^2 - i - \frac{1}{4}} \cdot 2^{i^2} \right]
\\
&\leq c \left[2^{-iq} + 2^{-i} \right]
\leq c(p,q) |I|.
\end{align*}

Let us say that a dyadic interval $I$ is long, if it is of the form $I_J$ for some long $J$.
Then we conclude that
\begin{equation*}
\sum_{\substack{J \in \mathcal{J} \\ J \text{  long  } }} F(J)
= \sum_{ I \text{  long  } }
\sum_{\substack{ J \text{  long  } \\ I = I_J }} F(J)
\leq \sum_{I \text{  long  }} 2 c |I|
\leq 2c
= c_4(p,q)
< \infty.
\end{equation*}
\end{proof}

Now we are ready to prove the following proposition, which also concludes the proof of Theorem \ref{päätulos}.
\begin{proposition}
$v = f^{1/q} \in JN_q(Q_0)$.
\end{proposition}
\begin{proof}
This follows from Propositions \ref{containprop}, \ref{shortprop}, \ref{mediumprop} and \ref{carlesonprop}. If we take any partition $\mathcal{J}$ of $Q_0$ into disjoint subintervals $J \in \mathcal{J}$, then
\begin{align*}
\sum_{J \in \mathcal{J}} |J| \left( \fint_{J} \left| v - v_J \right| \right)^q
&= \sum_{\substack{J \in \mathcal{J} \\ J \text{  contained  }}} F(J) + \sum_{\substack{J \in \mathcal{J} \\ J \text{  short  }}} F(J)
\\
&+ \sum_{\substack{J \in \mathcal{J} \\ J \text{  medium  }}} F(J) + \sum_{\substack{J \in \mathcal{J} \\ J \text{  long  }}} F(J)
\\
&\leq c_1 + c_2 + c_3 + c_4
< \infty.
\end{align*}
As the estimate does not depend on $\mathcal{J}$, we can conclude that $\left\| v \right\|_{JN_q(Q_0)}^q < \infty$
and therefore $v \in JN_q(Q_0)$.
\end{proof}

The function $f$ can be extended from the one-dimensional case $Q_0 \subset \mathbb{R}$ into the multidimensional case $Q_0 \subset \mathbb{R}^n$ with an arbitrary $n$, by repeatedly using Proposition \ref{laajennus}.

\begin{corollary}
Let $Q_0 \subset \mathbb{R}^n$ be a bounded cube and let $1 < p < \infty$. Then there exists a nonnegative function $f$ such that $f^{1/q} \notin JN_q(Q_0)$ whenever $1 < q \leq p$, and $f^{1/q} \in JN_q(Q_0)$ for every $q > p$.
\end{corollary}

\section{The vanishing subspace $VJN_p$}
\label{vjnposio}

The vanishing subspace $VJN_p$ of $JN_p$ is defined analogously to the space of vanishing mean oscillation $VMO$ which is a subspace of $BMO$.

\begin{definition}[$VJN_p$]
Let $1 < p < \infty$ and $Q_0 \subset \mathbb{R}^n$ a bounded cube. Then $f \in VJN_p(Q_0)$ if $f \in JN_p(Q_0)$ and
\begin{equation*}
\lim_{a \to 0}
\sup_{\substack{Q_i \subset Q_0 \\  l(Q_i) \leq a}}
\sum_{i=1}^{\infty} |Q_i| \left( \fint_{Q_i} | f - f_{Q_i} | \right)^p
= 0
\end{equation*}
where the supremum is taken over all collections of pairwise disjoint subcubes $Q_i \subset Q_0$ such that the side length of each $Q_i$ is at most $a$.
\end{definition}
The following theorem is a characterization of $VJN_p(Q_0)$.
\begin{theorem}
\label{vjnplause}
\begin{equation*}
VJN_p(Q_0)
= \overline{C^{\infty}(Q_0)}^{JN_p(Q_0)}
\end{equation*}
where $C^{\infty}(Q_0)$ is the set of smooth functions in $\mathbb{R}^n$ that have been restricted to $Q_0$.
\end{theorem}

For the proof we refer to \cite[Theorem 5.3]{vjnplahde}. To be precise, that theorem applies to the so-called John-Nirenberg-Campanato spaces $JN_{(p,q,s)_{\alpha}}(X)$ that were studied by Tao et al. \cite{campanato}. Here $p \in (1,\infty)$, $q \in [1,\infty)$, $\alpha \in [0,\infty)$, $s$ is a nonnegative integer and $X$ is either $\mathbb{R}^n$ or a bounded cube $Q_0 \subset \mathbb{R}^n$. Brudnyi and Brudnyi studied the vanishing subspace of this John-Nirenberg-Campanato space in \cite[Theorem 2.6]{brudnyi}, though they used some different notation. For consistency of this article, we denote the space by $VJN_{(p,q,s)_{\alpha}}(X)$ .

By choosing $q=1$, $s=\alpha=0$ and $X = Q_0$, the John-Nirenberg-Campanato space $JN_{(p,q,s)_{\alpha}}(X)$ becomes the John-Nirenberg space, $JN_{(p,1,0)_{0}}(Q_0) = JN_p(Q_0)$. The same holds for the vanishing subspaces, i.e. $VJN_p(Q_0) = VJN_{(p,1,0)_{0}}(Q_0)$.
So by choosing the parameters appropriately, \cite[Theorem 5.3]{vjnplahde} simplifies into the form in Theorem \ref{vjnplause}.

It follows directly from Theorem \ref{vjnplause} that $L^p(Q_0) \subset VJN_p(Q_0)$, by density of $C^{\infty}$ in $L^p$ and by Equation 
(\ref{hölder}).
Thus it is immediate that
\begin{equation*}
L^p(Q_0)
\subseteq VJN_p(Q_0)
\subseteq JN_p(Q_0).
\end{equation*}

The obvious question is now if these inclusions are strict. 
We answer this question by examining two functions.
The counterexample presented in \cite{nontriviality} is in $JN_p \setminus VJN_p$. By modifying the function, we get another function which is in $VJN_p \setminus L^p$.

\begin{proposition}
\label{jnpmuttaeivjnp}
Let $1 < p < \infty$ and $Q_0 \subset \mathbb{R}$ a bounded interval. Then there is a function $g : Q_0 \rightarrow \mathbb{R}$ such that $g \in JN_p(Q_0) \setminus VJN_p(Q_0)$.
\end{proposition}
\begin{proof}
Let $g$ be the same function as in Section \ref{pääosio}. It is already known that $g \in JN_p(Q_0)$. 
To show that $g \notin VJN_p(Q_0)$, let $a>0$ be arbitrarily small. Then there is a positive integer $m$ such that $2 l_m \leq a$. For any integer $k > m$ we notice that $d_{k-1} \geq l_k \geq D_k$. This means that for each interval $\hat{I}$ of length $l_k$ we can choose an interval $J$ of length $2 l_k$ such that it covers $\hat{I}$ and all of its descendants on one side and the interval $J$ doesn't intersect any other interval $\hat{I'}$. Then we get
\begin{equation*}
|J| \left( \fint_J |g - g_{J}| \right)^p
\geq |J|^{1-p} \left( \int_{\hat{I}} |g - g_{J}| \right)^p
= (2 l_k)^{1-p} \left( l_k (h_k - g_J) \right)^p.
\end{equation*}
On the other hand we know that
\begin{align*}
g_J
&= \fint_J g
= \frac{1}{2 l_k} \left( h_k l_k + \sum_{i=1}^{\infty} 2^{i-1} h_{k+i} l_{k+i} \right)
\\
&= h_k \left( \frac{1}{2} + \frac{1}{4} \sum_{i=1}^{\infty} 2^{ \left( \frac{1}{p} - 1 \right) (i^2 + 2 ki)} \right)
\\
&\leq h_k \left( \frac{1}{2} + \frac{1}{4} 2^{ \left( \frac{1}{p} - 1 \right) 2k} \sum_{i=1}^{\infty} 2^{ \left( \frac{1}{p} - 1 \right) i^2} \right)
\\
&\leq \frac{3}{4} h_k
\end{align*}
at least when $k$ is large enough. Thus for $k > m$ large enough, we have
\begin{align*}
|J| \left( \fint_J |g - g_{J}| \right)^p
&\geq (2 l_k)^{1-p} \left( l_k (h_k - g_J) \right)^p
\\
&\geq 2^{1-p} l_k \left( \frac{h_k}{4} \right)^p
\\
&= 2^{1 - 3p} \cdot 2^{-k^2 - k - 1/4} \cdot 2^{k^2}
\\
&= 2^{\frac{3}{4} - 3 p} \cdot 2^{-k}.
\end{align*}
Since there are $2^{k-1}$ many intervals $\hat{I}$ of length $l_k$, this means that
\begin{equation*}
\sup_{\substack{Q_i \subset Q_0 \\  l(Q_i) \leq a}}
\sum_{i=1}^{\infty} |Q_i| \left( \fint_{Q_i} | g - g_{Q_i} | \right)^p
\geq 2^{k-1} \cdot 2^{\frac{3}{4} - 3 p} \cdot 2^{-k}
= 2^{-\frac{1}{4} - 3 p}
\end{equation*}
and this is true for all $a>0$, so especially the limit as $a \to 0$ is positive. Thus by definition $g \notin VJN_p(Q_0)$.
\end{proof}

\begin{proposition}
\label{vjnpmuttaeilp}
Let $1 < p < \infty$ and $Q_0 \subset \mathbb{R}$ a bounded interval. Then there is a function $g_0 : Q_0 \rightarrow \mathbb{R}$ such that $g_0 \in VJN_p(Q_0) \setminus L^p(Q_0)$.
\end{proposition}

\begin{proof}

Let us construct $g_0$ by modifying the function $g$ from Section \ref{pääosio}. For $g_0$ we set the height of every tower as $h_i = \left( \frac{2^{i^2}}{i} \right)^{1/p}$.
The other parameters $l_i$ and $d_i$ are the same as in (\ref{gparametrit}).
Recall that the heights for $g$ were defined as $h_i = 2^{i^2/p}$.

The proof of $g_0 \in JN_p(Q_0)$ follows the same steps as the proof of $g \in JN_p(Q_0)$ in \cite{nontriviality}. Also we notice that
\begin{equation*}
\| g_0 \|_{L^p}^p
= \int_{Q_0} g_0^p
= \sum_{i=1}^{\infty} 2^{i-1} h_i^p l_i
= \sum_{i=1}^{\infty} 2^{i-1} \cdot \frac{2^{i^2}}{i} \cdot 2^{-i^2 - i - 1/4}
= 2^{-1-1/4} \sum_{i=1}^{\infty} \frac{1}{i}
= \infty,
\end{equation*}
thus $g_0 \notin L^p(Q_0)$.

To prove that $g_0 \in VJN_p$, let $a>0$ be a small enough number. Then there is some positive integer $k$ such that $\delta_k \geq a > \delta_{k+1}$. Let $\mathcal{J}$ be a countable collection of disjoint intervals such that $|J| \leq a$ for every $J \in \mathcal{J}$. Define
\begin{equation*}
F(J)
:= |J| \left( \fint_J \left| g_0 - (g_0)_J \right| \right)^p.
\end{equation*}
Similar to the function $u$ in Section \ref{pääosio}, it is clear that if $J$ does not intersect any of the intervals $\hat{I}$, then $F(J) = 0$. If $J$ does intersect some intervals, then there is a unique widest interval $I_J$, such that $J$ intersects $\hat{I}_J$. Also if $J$ does intersect an interval $\hat{I}$, if it does not intersect the boundary of the interval, then $F(J)=0$. Therefore we assume that $J$ intersects the boundary of $\hat{I}_J$. Then there are at most 2 intervals $J$ for every $I$ such that $I = I_J$. Also we can categorize the intervals $J$ as short, medium or long, depending on the length of $J \setminus \hat{I}_J$, the same way as in Definition \ref{containedshortmediumlong} and in \cite{nontriviality}.

We have the following bound for $F(J)$:
\begin{equation}
\label{fjmuttavjnp}
F(J)
\leq c \left[ |J|^{1-p} \left( \int_{J \setminus \hat{I} } g_0 \right)^p + \min \big( |J \setminus \hat{I}|,|J \cap \hat{I}| \big) h_I^p \right],
\end{equation}
where $c = c(p)$ and $I = I_J$. The proof is the same as for \cite[Lemma 3.6]{nontriviality}.

If $I = I_J$ and $|I| = 2^{-i}$ with $i \leq k$, then the fact that $|J \setminus \hat{I}| \leq |J| \leq a \leq \delta_k \leq \delta_i$ implies that $J$ is short. Hence in this case $F(J) \leq c \min \big( |J \setminus \hat{I}|,|J \cap \hat{I}| \big) h_I^p \leq c d_k h_k^p$. The total number of dyadic intervals $I$ that are not shorter than $2^{-k}$ is $1 + 2 + 4 + ... + 2^{k-1} \leq 2^{k}$. Therefore
\begin{equation*}
\sum_{\substack{J \in \mathcal{J} \\ |I_J| \geq 2^{-k} }} F(J)
\leq 2^{k} \cdot 2 c d_k h_k^p
= c 2^k \cdot 2^{-k^2 - 2k} \frac{2^{k^2}}{k}
\leq c 2^{-k}.
\end{equation*}
Since $k \to \infty$ when $a \to 0$, we notice that the limit of this sum is 0.

Assume from now on that $I = I_J$ and $i>k$. If $J$ is short, then $F(J) \leq c d_i h_i^p \leq c 2^{-i^2 - 2i} 2^{i^2} = c 2^{-2i}$. This implies that
\begin{equation*}
\sum_{\substack{J \in \mathcal{J} \\ |I_J| < 2^{-k} \\ J \text{  short  } }} F(J)
\leq \sum_{i=k+1}^{\infty} 2^i \cdot 2 c 2^{-2i}
= c \sum_{i=1}^{\infty} 2^{-k-i}
= c 2^{-k}
\end{equation*}
which tends to zero when $a \to 0$.

If $J$ is medium, then for the second term in (\ref{fjmuttavjnp}) we have essentially the same bound as in the short case. For the first term we have the bound
\begin{equation*}
|J|^{1-p} \left( \int_{J \setminus \hat{I}} g_0 \right)^p
\leq c(p) 2^{-ip}.
\end{equation*}
For the proof we refer to \cite[Lemma 3.8]{nontriviality}, because $g_0 \leq g$ pointwise. This implies that
\begin{align*}
\sum_{\substack{J \in \mathcal{J} \\ |I_J| < 2^{-k} \\ J \text{  medium  } }} F(J)
&\leq \sum_{i=k+1}^{\infty} 2^i \cdot 2 c (2^{-2i} + 2^{-ip})
= c \sum_{i=1}^{\infty} 2^{-k-i} + 2^{(1-p)(k+i)}
\\
&\leq c \left( 2^{-k} + 2^{(1-p)k} \right).
\end{align*}
This also tends to zero when $a \to 0$.

Finally if $J$ is long, we have the same Carleson property as in Proposition \ref{carlesonprop} and in \cite[Lemma 3.9]{nontriviality}:
\begin{equation*}
\sum_{\substack{I \text{  dyadic  } \\ \exists J \text{  long s.t.  } I = I_J}} |I| \leq 1.
\end{equation*}
The bound for the first term in (\ref{fjmuttavjnp}) is the same as in the medium case:
\begin{equation*}
|J|^{1-p} \left( \int_{J \setminus \hat{I}} g_0 \right)^p
\leq c(p) 2^{-ip}
= c 2^{(1-p)i} |I|.
\end{equation*}
For the second term we have the bound
\begin{equation*}
\min \big( |J \setminus \hat{I}|,|J \cap \hat{I}| \big) h_I^p
\leq l_i h_i^p
= 2^{-i^2 - i - 1/4} \frac{2^{i^2}}{i}
\leq \frac{|I|}{i}.
\end{equation*}
This implies that
\begin{align*}
\sum_{\substack{J \in \mathcal{J} \\ |I_J| < 2^{-k} \\ J \text{  long  } }} F(J)
&= \sum_{\substack{I \text{  long  } \\ |I| < 2^{-k} }}
\sum_{\substack{ J \text{  long  } \\ I = I_J }}
F(J)
\leq \sum_{\substack{I \text{  long  } \\ |I| < 2^{-k}  }} 2c \left( 2^{(1-p)i} |I| + \frac{|I|}{i} \right)
\\
&\leq c \left( 2^{(1-p)k} + \frac{1}{k} \right) \sum_{\substack{I \text{  long  } }} |I|
\leq c \left( 2^{(1-p)k} + \frac{1}{k} \right).
\end{align*}
This also tends to zero.

In conclusion
\begin{align*}
\sum_{J \in \mathcal{J}} |J| \left( \fint_{J} \left| g_0 - (g_0)_J \right| \right)^p
&= \sum_{\substack{J \in \mathcal{J} \\ |I_J| \geq 2^{-k} }} F(J) + \sum_{\substack{J \in \mathcal{J} \\ |I_J| < 2^{-k} \\ J \text{  short  } }} F(J)
\\
&+ \sum_{\substack{J \in \mathcal{J} \\ |I_J| < 2^{-k} \\ J \text{  medium  } }} F(J) + \sum_{\substack{J \in \mathcal{J} \\ |I_J| < 2^{-k} \\ J \text{  long  } }} F(J)
\\
&\leq c 2^{-k} + c 2^{-k} + c 2^{-k} + c 2^{(1-p)k} + c 2^{(1-p)k} + \frac{c}{k}
\end{align*}
and this holds for any admissible collection $\mathcal{J}$. This implies that
\begin{equation*}
\lim_{a \to 0}
\sup_{\substack{Q_i \subset Q_0 \\  l(Q_i) \leq a}}
\sum_{i=1}^{\infty} |Q_i| \left( \fint_{Q_i} | g_0 - (g_0)_{Q_i} | \right)^p
\leq \lim_{k \to \infty} c \left( 2^{-k} + 2^{(1-p)k} + \frac{1}{k} \right)
= 0
\end{equation*}
and so by definition $g_0 \in VJN_p(Q_0)$.
\end{proof}

It is worth noting that the choice of $h_i$ in the previous proof does not need to be exactly what it is. It is sufficient that $h_i = \left( 2^{i^2} \cdot s_i \right)^{1/p}$ where $s_i$ is a sequence of nonnegative real numbers such that
\begin{equation*}
\lim_{i \to \infty} s_i = 0
\text{  and  }
\sum_{i=1}^{\infty} s_i = \infty.
\end{equation*}

Now we would like to extend this result to the multidimensional case. The following proposition is a counterpart of Proposition \ref{laajennus}.

\begin{proposition}
\label{laajennus2}
Let $Q_0 \subset \mathbb{R}^n$ be a bounded cube, $f \in L^1(Q_0)$, and $\tilde{f}(x,t) := f(x)$ its trivial extension to $(x,t) \in \tilde{Q}_0 := Q_0 \times [0,l(Q_0)) \subset \mathbb{R}^{n+1}$.
Then $f \in VJN_p(Q_0)$ if and only if $\tilde{f} \in VJN_p(\tilde{Q}_0)$.
\end{proposition}

\begin{proof}
The proof is very similar to the proof of \cite[Proposition 4.1]{nontriviality}. The only difference is that in the partition of $Q_0$ into subcubes we have the additional criterion that $l(Q_i) \leq a$ for every $i$ and for some $a > 0$. Ultimately we get the result by taking the limit $a \to 0$.
\end{proof}

By using Propositions \ref{laajennus} and \ref{laajennus2} multiple times, we can extend both $g$ and $g_0$ into the multidimensional cube $Q_0 \subset \mathbb{R}^n$ for any $n$, such that $g \in JN_p(Q_0) \setminus VJN_p(Q_0)$ and $g_0 \in VJN_p(Q_0) \setminus L^p(Q_0)$. In conclusion we have shown that for any bounded cube $Q_0 \subset \mathbb{R}^n$
\begin{equation*}
L^p(Q_0)
\subsetneq VJN_p(Q_0)
\subsetneq JN_p(Q_0)
\subsetneq L^{p,\infty}(Q_0).
\end{equation*}

\section{The spaces $JN_p$, $VJN_p$ and $CJN_p$ defined on $\mathbb{R}^n$}
\label{rn}

The definition of the space $JN_p$ can be extended from the bounded cube $Q_0$ into the whole Euclidian space $\mathbb{R}^n$. In this section we deal with the entire space $\mathbb{R}^n$ and always assume that $JN_p = JN_p(\mathbb{R}^n)$ and $VJN_p = VJN_p(\mathbb{R}^n)$ etc. unless otherwise specified.
\begin{definition}[$JN_p$]
\label{jnprn}
Let $1 \leq p < \infty$. Then $f \in JN_p$ if $f \in L_{loc}^1$ and there is a constant $K < \infty$ such that
\begin{equation*}
\sum_{i=1}^{\infty} |Q_i| \left( \fint_{Q_i} | f - f_{Q_i} | \right)^p \leq K^p
\end{equation*}
for all countable collections of pairwise disjoint cubes $(Q_i)_{i=1}^{\infty}$ in $\mathbb{R}^n$.
We denote the smallest such number $K$ by $\|f\|_{JN_p}$.
\end{definition}

The embedding $JN_p(Q_0) \subset L^{p,\infty}(Q_0)$ in Theorem \ref{embedding} has been proven in a number of ways. However all the proofs, that we know of, only consider $JN_p$ on bounded sets, not on $\mathbb{R}^n$.
Clearly any constant function $f=c$ is in the space $JN_p$ even though it is in $L^{p,\infty}$ only if $c=0$. So let us define for $p > 1$ the space
\begin{equation*}
L^{p,\infty} / \mathbb{R}
:= \left\{ f \in L_{loc}^1 : \exists \text{  } c \in \mathbb{R} \text{  s.t.  } f - c \in L^{p,\infty} \right\}.
\end{equation*}
Then for the sake of completeness we prove the following theorem.
\begin{theorem}
Let $1 < p < \infty$ and $f \in JN_p$. Then $f \in L^{p,\infty} / \mathbb{R}$ and there is a constant $b$ such that
\begin{equation*}
\| f - b \|_{L^{p,\infty}}
\leq c \| f \|_{JN_p}
\end{equation*}
where $c = c(n,p)$.
\end{theorem}

\begin{proof}
First let $Q \subset \mathbb{R}^n$ be any bounded cube. It is clear that if we restrict $f$ to this cube, then we have $f \in JN_p(Q)$ and therefore from Theorem \ref{embedding} we get that $\| f - f_Q \|_{L^{p,\infty}(Q)} \leq c \| f \|_{JN_p(Q)}$. Further from Definition \ref{jnprn} it immediately follows that
\begin{equation*}
\fint_Q |f - f_Q|
\leq |Q|^{-1/p} \| f \|_{JN_p}.
\end{equation*}
Now let us define a sequence of cubes $(Q_m)_{m=1}^{\infty}$ such that $|Q_m| = 2^m$ and the center of every cube is the origin. Then clearly we have $Q_1 \subset Q_2 \subset ...$ and $\bigcup_{m=1}^{\infty} Q_m = \mathbb{R}^n$.
Let us prove that the sequence of integral averages $f_{Q_m}$ is a Cauchy sequence. 
First we notice that for the difference of two consecutive elements we have the following estimate:
\begin{align*}
|f_{Q_{i+1}} - f_{Q_{i}}|
&\leq \frac{|Q_{i+1}|}{|Q_i|} \fint_{Q_{i+1}} |f_{Q_{i+1}} - f| + \fint_{Q_i} |f - f_{Q_{i}}|
\\
&\leq 2 |Q_{i+1}|^{-1/p} \| f \|_{JN_p} + |Q_i|^{-1/p} \| f \|_{JN_p}
\\
&= \left( 2^{1-1/p} + 1 \right) \| f \|_{JN_p} 2^{-i/p}.
\end{align*}
Now let $\epsilon > 0$ and $N = N(\epsilon)$ an integer that will depend on $\epsilon$. Let $m , k \geq N$ and let us assume that $k \geq m$. Then
\begin{align*}
|f_{Q_k} - f_{Q_m}|
&\leq \sum_{i=m}^{k-1} |f_{Q_{i+1}} - f_{Q_{i}}|
\leq \sum_{i=N}^{\infty} \left( 2^{1-1/p} + 1 \right) \| f \|_{JN_p} 2^{-i/p}
\\
&= c(p) \| f \|_{JN_p} 2^{-N/p}
< \epsilon
\end{align*}
when $N$ is large enough. Thus $f_{Q_m}$ is a Cauchy sequence and we can set $b := \lim_{m \to \infty} f_{Q_m}$. 
Then we have for any positive integer $m$
\begin{align*}
\left| \left\{ x \in Q_m : |f(x) - b| > t \right\} \right|
&\leq \left| \left\{ x \in Q_m : |f(x) - f_{Q_m}| > \frac{t}{2} \right\} \right|
\\
&+ \left| \left\{ x \in Q_m : |f_{Q_m} - b| > \frac{t}{2} \right\} \right|
\\
&\leq c \frac{\| f \|_{JN_p(Q_m)}^p}{\left( \frac{t}{2} \right)^p} + \left| \left\{ x \in Q_m : |f_{Q_m} - b| > \frac{t}{2} \right\} \right|
\\
&\leq c \frac{\| f \|_{JN_p}^p}{t^p} + \left| \left\{ x \in Q_m : |f_{Q_m} - b| > \frac{t}{2} \right\} \right|
\end{align*}
for all $t > 0$. 
The second term is either equal to $|Q_m|$ or 0, but when we take the limit $m \to \infty$ it tends to 0. Therefore we have
\begin{equation*}
\left| \left\{ x \in \mathbb{R}^n : |f(x) - b| > t \right\} \right|
= \lim_{m \to \infty} \left| \left\{ x \in Q_m : |f(x) - b| > t \right\} \right|
\leq c \frac{\| f \|_{JN_p}^p}{t^p}.
\end{equation*}
This completes the proof.
\end{proof}

Notice also that the zero-extension of the function we had in Section \ref{preliminaries} in $L^{p,\infty}(Q_0) \setminus JN_p(Q_0)$ is in $L^{p,\infty} / \mathbb{R} \setminus JN_p$. This shows that $JN_p \subsetneq L^{p,\infty} / \mathbb{R}$.
It is easy to see, by using a similar method as in the previous proof, that $JN_1 = L^1 / \mathbb{R}$. Therefore we assume that $p > 1$ from now on.

The vanishing subspace $VJN_p$ can be defined in the whole space as well. 
\begin{definition}[$VJN_p$]
Let $1 < p < \infty$. Then the vanishing subspace $VJN_p$ of $JN_p$ is defined by setting
\begin{equation*}
VJN_p
:= \overline{D_p(\mathbb{R}^n) \cap JN_p}^{JN_p},
\end{equation*}
where
\begin{equation*}
D_p(\mathbb{R}^n)
:= \left\{ f \in C^{\infty}(\mathbb{R}^n) : |\nabla f| \in L^p(\mathbb{R}^n) \right\} .
\end{equation*}
\end{definition}

The following characterization of $VJN_p$ is proven in \cite[Theorem 3.2]{vjnplahde}.
\begin{theorem}
\label{vjnprnlause}
Let $1 < p < \infty$. Then $f \in VJN_p$ if and only if $f \in JN_p$ and
\begin{equation*}
\lim_{a \to 0}
\sup_{\substack{Q_i \subset \mathbb{R}^n \\  l(Q_i) \leq a}}
\sum_{i=1}^{\infty} |Q_i| \left( \fint_{Q_i} | f - f_{Q_i} | \right)^p
= 0
\end{equation*}
where the supremum is taken over all collections of pairwise disjoint cubes $Q_i$ such that the side length of each $Q_i$ is at most $a$.
\end{theorem}

We also study $CJN_p$, which is another subspace of $JN_p$. 
This space is defined analogously to the space of continuous mean oscillation $CMO$, which is a subspace of $BMO$.
In the case of the bounded cube, $CJN_p(Q_0) = VJN_p(Q_0)$. Hence we only define $CJN_p$ in the whole space.
\begin{definition}[$CJN_p$]
Let $1 < p < \infty$. Then the subspace $CJN_p$ of $JN_p$ is defined by setting
\begin{equation*}
CJN_p
:= \overline{C_c^{\infty}(\mathbb{R}^n) }^{JN_p},
\end{equation*}
where $C_c^{\infty}(\mathbb{R}^n)$ denotes the set of smooth functions with compact support in $\mathbb{R}^n$.
\end{definition}

The following characterization of $CJN_p$ is proven in \cite[Theorem 4.3]{vjnplahde}.
\begin{theorem}
\label{cjnplause}
Let $1 < p < \infty$. Then $f \in CJN_p$ if and only if $f \in VJN_p$ and
\begin{equation*}
\lim_{a \to \infty}
\sup_{\substack{Q \subset \mathbb{R}^n \\  l(Q) \geq a}}
|Q|^{1/p} \fint_Q |f - f_Q|
= 0
\end{equation*}
where the supremum is taken over all cubes $Q \subset \mathbb{R}^n$ such that the side length of $Q$ is at least $a$.
\end{theorem}

It follows directly from the definition of $CJN_p$ that $L^p \subset CJN_p$, by density of $C_c^{\infty}$ in $L^p$ and by Equation (\ref{hölder}).
Furthermore we notice from Theorems \ref{vjnprnlause} and \ref{cjnplause} and from Definition \ref{jnprn} that adding a constant to a function doesn't affect whether the function is in $JN_p$ or in the subspaces $VJN_p$ or $CJN_p$. Thus if $f \in L^p$, then $f+c \in CJN_p$ for any constant $c$. If we define $L^p / \mathbb{R}$ the same way as $L^{p,\infty} / \mathbb{R}$, then it is immediate that
\begin{equation*}
L^p / \mathbb{R}
\subseteq CJN_p
\subseteq VJN_p
\subseteq JN_p.
\end{equation*}

Again we can ask whether these inclusions are strict. 
We examine this by extending the previous functions into multidimensional versions in the whole space.

\begin{lemma}
\label{ekalisäys}
Let $I \subset \mathbb{R}$ be a bounded interval such that $|I| = L$. Without loss of generality we may assume that $I = [0,L]$. Let $f \in L^1(I)$, $n \in \mathbb{Z}_+$ and $1 < p < \infty$. Define the cube $Q_0 := [0,L]^{n} \subset \mathbb{R}^n$. Define the function
\begin{equation*}
\hat{f}(x_1 , x_2 , ... , x_n) := \begin{cases}
f(x_1), \text{  if  } 0 \leq x_i \leq L \text{  for all  } i \geq 1
\\
f_I \text{  elsewhere.}
\end{cases}
\end{equation*}
Then $\hat{f} \in JN_p$ if and only if $f \in JN_p(I)$ and
\begin{equation*}
2^{\frac{1-n}{p}} \| f \|_{JN_p(I)} L^{\frac{n-1}{p}}
\leq \big\| \hat{f} \big\|_{JN_p}
\leq c(n,p) \| f \|_{JN_p(I)} L^{\frac{n-1}{p}}.
\end{equation*}
\end{lemma}

\begin{proof}
If $\hat{f} \in JN_p$, then clearly if we restrict $\hat{f}$ to $Q_0$, we have $\hat{f} \in JN_p(Q_0)$. Then Proposition \ref{laajennus} implies that
\begin{equation*}
2^{\frac{1-n}{p}} \| f \|_{JN_p(I)} L^{\frac{n-1}{p}}
\leq \big\| \hat{f} \big\|_{JN_p(Q_0)}
\leq \| \hat{f} \|_{JN_p}
\end{equation*}
and so $f \in JN_p(I)$.

Now assume that $f \in JN_p(I)$. Without loss of generality we may assume that $f_I = 0$. Let $Q \subset \mathbb{R}^n$ be a cube such that $Q \cap Q_0 \neq \varnothing$ and $Q \setminus Q_0 \neq \varnothing$. Since $Q \cap Q_0$ is an intersection of two cubes, it is a rectangle. Therefore we can denote $Q \cap Q_0 = J_1 \times J_2 \times ... \times J_n = J_1 \times K$, where $J_i$ are intervals such that $J_i \subset [0,L]$ for every $i \geq 1$. The notation $|\cdot|$ for sets below refers to the Lebesgue measure, which can be either 1-, $n$- or $n-1$-dimensional depending on the context.
Then
\begin{align*}
|Q| \left( \fint_{Q} \left| \hat{f} - \hat{f}_Q \right| \right)^p
&\leq 2^p |Q \cap Q_0|^{1-p} \left( \int_{Q \cap Q_0} \big| \hat{f} \big| \right)^p
\\
&= 2^p |J_1|^{1-p} |K|^{1-p} \left( \int_{K} \int_{J_1} |f(x_1)| dx_1 dx_2 ... dx_n \right)^p
\\
&= 2^p |K| |J_1|^{1-p} \left( \int_0^{\infty} |\{ x \in J_1 : |f(x)| > t \}| dt \right)^p
\\
&\leq 2^p |K| |J_1|^{1-p} \Bigg( \int_0^{|J_1|^{-1/p} \| f \|_{L^{p,\infty}(J_1)}} |J_1| dt
\\
&+ \int_{|J_1|^{-1/p} \| f \|_{L^{p,\infty}(J_1)}}^{\infty} \frac{\| f \|_{L^{p,\infty}(J_1)}^p}{t^p} dt \Bigg)^p
\\
&= 2^p |K| |J_1|^{1-p} \Bigg( |J_1|^{1 - 1/p} \| f \|_{L^{p,\infty}(J_1)}
\\
&+ \| f \|_{L^{p,\infty}(J_1)}^p \frac{1}{p-1} (|J_1|^{-1/p} \| f \|_{L^{p,\infty}(J_1)})^{1-p} \Bigg)^p
\\
&\leq \left( \frac{2p}{p-1} \right)^p \| f \|_{L^{p,\infty}(I)}^p |K|
\\
&\leq c(p) \| f \|_{JN_p(I)}^p |K|.
\end{align*}
Here we used the triangle inequality, Cavalieri's principle and Theorem \ref{embedding}. For the constant we have $\lim_{p \to 1} c(p) = \infty$, even though the result would clearly hold also, if $p = 1$.
Therefore a better constant could perhaps be attained in this inequality. However we were not able to produce that.

Let us look at $Q \cap Q_0 = J_1 \times ... \times J_n$. Clearly we have $|J_i| \leq l(Q)$ for every $i$. Further we know that $|J_j| < l(Q)$ for at least one $j$ - otherwise we would have $|Q \cap Q_0| = |Q|$.

If $|J_1| < l(Q)$, then the cube $Q$ intersects with the boundary of $I$. This implies that $K$, the $n-1$-dimensional side of $Q \cap Q_0$, is located at the boundary of $Q_0$. Therefore we can estimate
\begin{equation*}
|K|
\leq |\{ x \in Q \cap Q_0 : x_1 \in \partial I \}|
\leq |\{ x \in Q \cap Q_0 : x \in \partial Q_0 \}|
= |Q \cap \partial Q_0|.
\end{equation*}
Here the first inequality is an equality if $Q$ intersects only one of the boundary points of $I$. If it intersects the other boundary point, we have $|\{ x \in Q \cap Q_0 : x_1 \in \partial I \}| = 2 |K|$.

If $|J_1| = l(Q)$, then $|J_i| < l(Q)$ for some $i \geq 2$. By the same reasoning as before we know that $Q$ intersects the boundary of the $i$:th dimensional interval $[0,L]$ and therefore
\begin{align*}
|K|
&= |J_2| |J_3| ... |J_n|
\leq |J_1| |J_2| ... |J_{i-1}| |J_{i+1}| ... |J_n|
\\
&\leq |\{ x \in Q \cap Q_0 : x_i \in \partial [0,L] \}|
\leq |\{ x \in Q \cap Q_0 : x \in \partial Q_0 \}|
\\
&= |Q \cap \partial Q_0|.
\end{align*}
Thus in any case we have $|K| \leq |Q \cap \partial Q_0|$. Now let $\mathcal{Q}$ be a partition of $\mathbb{R}^n$ into disjoint cubes and define
\begin{equation*}
F(Q)
:= |Q| \left( \fint_Q |\hat{f} - \hat{f}_Q| \right)^p. 
\end{equation*}
Then
\begin{align*}
\sum_{Q \in \mathcal{Q}} F(Q)
&= \sum_{\substack{Q \in \mathcal{Q} \\ Q \subset Q_0}} F(Q) + \sum_{\substack{Q \in \mathcal{Q} \\ Q \subset \mathbb{R}^n \setminus Q_0}} F(Q) + \sum_{\substack{Q \in \mathcal{Q} \\ Q \cap Q_0 \neq \varnothing \\ Q \setminus Q_0 \neq \varnothing }} F(Q)
\\
&\leq \big\| \hat{f} \big\|_{JN_p(Q_0)}^p + 0 + \sum_{\substack{Q \in \mathcal{Q} \\ Q \cap Q_0 \neq \varnothing \\ Q \setminus Q_0 \neq \varnothing }} c(p) \| f \|_{JN_p(I)}^p |K|
\\
&\leq \| f \|_{JN_p(I)}^p L^{n-1} + c(p) \| f \|_{JN_p(I)}^p \sum_{\substack{Q \in \mathcal{Q} \\ Q \cap Q_0 \neq \varnothing \\ Q \setminus Q_0 \neq \varnothing }} |Q \cap \partial Q_0|
\\
&\leq \| f \|_{JN_p(I)}^p \left( L^{n-1} + c(p) |\partial Q_0| \right)
\\
&= \| f \|_{JN_p(I)}^p \left( L^{n-1} + c(p) 2 n L^{n-1} \right)
\\
&= c(n,p) \| f \|_{JN_p(I)}^p L^{n-1}.
\end{align*}
Here we used Proposition \ref{laajennus}. Finally by taking the supremum this gives us the result
\begin{equation*}
\big\| \hat{f} \big\|_{JN_p}
\leq c(n,p) \| f \|_{JN_p(I)} L^{\frac{n-1}{p}}.
\end{equation*}
\end{proof}

The moral is that if we have any function in $JN_p(I)$ for a bounded interval $I$, we can extend the function into multiple dimensions and into the whole space so that we get a function in $JN_p(\mathbb{R}^n)$.
It is easy to see that $\hat{f} \in L^p(\mathbb{R}^n) / \mathbb{R}$ if and only if $f \in L^p(I)$ and $\hat{f} \in L^{p,\infty}(\mathbb{R}^n) / \mathbb{R}$ if and only if $f \in L^{p,\infty}(I)$.
The same results hold also if $n=1$.

\begin{lemma}
\label{tokalisäys}
Let $I \subset \mathbb{R}$ be a bounded interval such that $|I| = L$. Without loss of generality we may assume that $I = [0,L]$. Let $f \in L^1(I)$, $n \in \mathbb{Z}_+$ and $1 < p < \infty$. Define the cube $Q_0 := [0,L]^{n} \subset \mathbb{R}^n$. Define the function
\begin{equation*}
\hat{f}(x_1 , x_2 , ... , x_n) := \begin{cases}
f(x_1), \text{  if  } 0 \leq x_i \leq L \text{  for all  } i \geq 1
\\
f_I \text{  elsewhere.}
\end{cases}
\end{equation*}
Then if $\hat{f} \in VJN_p$, we have $f \in VJN_p(I)$.
On the other hand if $f \in VJN_p(I)$ and we assume further that
\begin{equation}
\label{auki}
\lim_{a \to 0}
\sup_{\substack{J \subset I \\ |J| \leq a }} |J|^{1-p} \left( \int_{J} |f| \right)^p
= 0,
\end{equation}
where the supremum is taken over all intervals $J$ such that $|J| \leq a$, then $\hat{f} \in CJN_p$.
\end{lemma}

\begin{proof}
If $\hat{f} \in VJN_p$ and we restrict $\hat{f}$ to $Q_0$, we clearly have $\hat{f} \in VJN_p(Q_0)$. Then Proposition \ref{laajennus2} implies that $f \in VJN_p(I)$.


Now assume that $f \in VJN_p(I)$ and assumption (\ref{auki}) holds. Without loss of generality we may assume that $f_I = 0$. First we notice that then $f \in JN_p(I)$, which means that $\hat{f} \in JN_p$, by Lemma \ref{ekalisäys}.
Now let $a > 0$. Let $Q \subset \mathbb{R}^n$ be a cube such that $Q \cap Q_0 \neq \varnothing$, $Q \setminus Q_0 \neq \varnothing$ and $l(Q) \leq a$.
Similar to the previous proof, we can write $Q \cap Q_0 = J_1 \times J_2 \times ... \times J_n = J_1 \times K$. Then
\begin{equation*}
|Q| \left( \fint_{Q} \left| \hat{f} - \hat{f}_Q \right| \right)^p
\leq 2^p |K| |J_1|^{1-p} \left( \int_{J_1} |f| \right)^p
\end{equation*}
as in the previous proof. Further we have $|J_1| \leq l(Q) \leq a$, which means that
\begin{equation*}
|J_1|^{1-p} \left( \int_{J_1} |f| \right)^p
\leq \sup_{\substack{J \subset I \\ |J| \leq a }} |J|^{1-p} \left( \int_{J} |f| \right)^p.
\end{equation*}
As in the previous proof, we also have $|K| \leq |Q \cap \partial Q_0|$. Now let $\mathcal{Q}$ be a partition of $\mathbb{R}^n$ into disjoint cubes such that the side length of each cube is at most $a$. Define
\begin{equation*}
F(Q)
:= |Q| \left( \fint_Q |\hat{f} - \hat{f}_Q| \right)^p. 
\end{equation*}
Then
\begin{align*}
\sum_{\substack{Q \in \mathcal{Q} \\ Q \cap Q_0 \neq \varnothing \\ Q \setminus Q_0 \neq \varnothing }} F(Q)
&\leq \sum_{\substack{Q \in \mathcal{Q} \\ Q \cap Q_0 \neq \varnothing \\ Q \setminus Q_0 \neq \varnothing }} 2^p |K| \sup_{\substack{J \subset I \\ |J| \leq a }} |J|^{1-p} \left( \int_{J} |f| \right)^p
\\
&\leq 2^p \sup_{\substack{J \subset I \\ |J| \leq a }} |J|^{1-p} \left( \int_{J} |f| \right)^p \sum_{\substack{Q \in \mathcal{Q} \\ Q \cap Q_0 \neq \varnothing \\ Q \setminus Q_0 \neq \varnothing }} |Q \cap \partial Q_0|
\\
&\leq 2^p \sup_{\substack{J \subset I \\ |J| \leq a }} |J|^{1-p} \left( \int_{J} |f| \right)^p |\partial Q_0|
\\
&= 2^p \cdot 2 n L^{n-1} \sup_{\substack{J \subset I \\ |J| \leq a }} |J|^{1-p} \left( \int_{J} |f| \right)^p.
\end{align*}
Therefore
\begin{align*}
\sum_{Q \in \mathcal{Q}} F(Q)
&= \sum_{\substack{Q \in \mathcal{Q} \\ Q \subset Q_0}} F(Q) + \sum_{\substack{Q \in \mathcal{Q} \\ Q \subset \mathbb{R}^n \setminus Q_0}} F(Q) + \sum_{\substack{Q \in \mathcal{Q} \\ Q \cap Q_0 \neq \varnothing \\ Q \setminus Q_0 \neq \varnothing }} F(Q)
\\
&\leq \sup_{\mathcal{Q}} \sum_{\substack{Q \in \mathcal{Q} \\ Q \subset Q_0 }} F(Q) + 0 + 2^p \cdot 2 n L^{n-1} \sup_{\substack{J \subset I \\ |J| \leq a }} |J|^{1-p} \left( \int_{J} |f| \right)^p.
\end{align*}
If we restrict $\hat{f}$ to $Q_0$, then Proposition \ref{laajennus2}, together with $f \in VJN_p(I)$, implies that $\hat{f} \in VJN_p(Q_0)$. This gives us that
\begin{equation*}
\lim_{a \to 0}
\sup_{\mathcal{Q}} \sum_{\substack{Q \in \mathcal{Q} \\ Q \subset Q_0 }} F(Q)
= 0.
\end{equation*}
For the other term we have a similar limit by assumption (\ref{auki}).
Then we have
\begin{equation*}
\lim_{a \to 0}
\sup_{\mathcal{Q}}
\sum_{Q \in \mathcal{Q}} F(Q)
= 0
\end{equation*}
and so Theorem \ref{vjnprnlause} gives us that $\hat{f} \in VJN_p$.
Finally we know that $f \in L^1(I)$ and $\hat{f}$ is supported in $Q_0$. Therefore
\begin{align*}
\lim_{a \to \infty}
\sup_{\substack{Q \subset \mathbb{R}^n \\ l(Q) \geq a}}
|Q|^{1/p} \fint_Q \left| \hat{f} - \hat{f}_Q \right|
&\leq \lim_{a \to \infty}
\sup_{\substack{Q \subset \mathbb{R}^n \\ l(Q) \geq a}}
|Q|^{1/p} 2 \fint_Q \big| \hat{f} \big|
\\
&\leq \lim_{a \to \infty}
\sup_{\substack{Q \subset \mathbb{R}^n \\ l(Q) \geq a}}
|Q|^{1/p - 1} 2 \int_{Q_0} \big| \hat{f} \big|
\\
&= 2 L^{n-1} \int_{I} |f| \lim_{a \to \infty}
(a^n)^{1/p - 1}
= 0,
\end{align*}
and so Theorem \ref{cjnplause} gives us that $\hat{f} \in CJN_p$. This completes the proof.
\end{proof}

\begin{lemma}
\label{kolmaslisäys}
Let $g_0$ be the function in Proposition \ref{vjnpmuttaeilp} such that $g_0 \in VJN_p(Q_0) \setminus L^p(Q_0)$ with $Q_0 \subset \mathbb{R}$. Then we have
\begin{equation*}
\lim_{a \to 0}
\sup_{\substack{J \subset Q_0 \\ |J| \leq a }} |J|^{1-p} \left( \int_{J} g_0 \right)^p
= 0,
\end{equation*}
where the supremum is taken over all intervals $J$ such that $|J| \leq a$.
\end{lemma}

\begin{proof}
It is enough that we prove this result for the function $g$ from Proposition \ref{jnpmuttaeivjnp}, because clearly $g_0 \leq g$.
Let $a$ be a small positive number such that $\delta_{N} < a \leq \delta_{N-1}$ for some integer $N \geq 2$. Let $J \subset Q_0$ be an interval such that $|J| \leq a$. Then there is some integer $m \geq N$ such that $\delta_m < |J| \leq \delta_{m-1}$. Like in earlier cases, if $J$ does not intersect any interval $\hat{I}$, then
\begin{equation*}
|J|^{1-p} \left( \int_{J} g \right)^p
= 0.
\end{equation*}
Therefore we assume that $J$ does intersect some interval $\hat{I}$. Then there is a unique widest interval $I_J$ such that $J$ intersects $\hat{I}_J$.

First assume that $|I_J| = |I| = 2^{-i} > 2^{-m}$. Then the fact that $|J \setminus \hat{I}| \leq |J| \leq \delta_{m-1} \leq \delta_i$ implies that $\hat{I}$ is the only interval that $J$ intersects. Then
\begin{align*}
|J|^{1-p} \left( \int_{J} g \right)^p
&\leq |J|^{1-p} \left( \int_{J} h_i \right)^p
= |J| h_i^p
\leq \delta_{m-1} h_{m-1}^p
\leq 2^{-m^2} \cdot 2^{(m-1)^2}
\\
&= 2^{1 - 2m}
\leq 2^{-N}.
\end{align*}
As $a \to 0$, we know that $N \to \infty$. Therefore this upper bound converges to 0.

Assume now that $|I_J| = 2^{-i} \leq 2^{-m}$. Then $J$ intersects at most one interval of length $|\hat{I}| = l_m$, two intervals of length $|\hat{I}| = l_{m+1}$ etc. If $I$ is an interval of length $2^{-m}$, then in this case we have
\begin{equation*}
\int_J g
\leq \min(|J|,l_m) h_m + \sum_{\substack{I' \subsetneq I \\ I' \text{  dyadic  }}} \int_{\hat{I'}} g.
\end{equation*}
For the second term we have the following upper bound.
\begin{equation*}
|J|^{1-p} \left( \sum_{\substack{I' \subsetneq I \\ I' \text{  dyadic  }}} \int_{\hat{I'}} g \right)^p
\leq c 2^{-mp}.
\end{equation*}
For the proof we refer to \cite[Lemma 3.8]{nontriviality}.
For the first term we notice that
\begin{equation*}
|J|^{1-p} \min(|J|,l_m)^p
\leq l_m
\end{equation*}
whether $|J| < l_m$ or $l_m \leq |J|$.
This means that
\begin{align*}
|J|^{1-p} \left( \int_J g \right)^p
&\leq c |J|^{1-p} \left( \min(|J|,l_m) h_m \right)^p + c |J|^{1-p} \left( \sum_{\substack{I' \subsetneq I \\ I' \text{  dyadic  }}} \int_{\hat{I'}} g \right)^p
\\
&\leq c l_m h_m^p + c 2^{-mp}
\leq c 2^{-m^2-m} \cdot 2^{m^2} + c 2^{-m}
\leq c 2^{-N}.
\end{align*}

Hence whenever $|J| \leq a$, we have
\begin{equation*}
|J|^{1-p} \left( \int_J g \right)^p
\leq c(p) 2^{-N}.
\end{equation*}
This means that
\begin{equation*}
\lim_{a \to 0}
\sup_{\substack{J \subset Q_0 \\ |J| \leq a }} |J|^{1-p} \left( \int_{J} g \right)^p
\leq \lim_{N \to \infty} c(p) 2^{-N}
= 0
\end{equation*}
and therefore functions $g$ and $g_0$ satisfy condition (\ref{auki}).
\end{proof}

Now Lemmas \ref{ekalisäys}, \ref{tokalisäys} and \ref{kolmaslisäys} together with the examples $g \in JN_p(Q_0) \setminus VJN_p(Q_0)$ from Proposition \ref{jnpmuttaeivjnp} and $g_0 \in VJN_p(Q_0) \setminus L^p(Q_0)$ from Proposition \ref{vjnpmuttaeilp}, imply the following corollary.

\begin{corollary}
For any $n \in \mathbb{Z}_+$ there exists a function $\hat{g} \in JN_p(\mathbb{R}^n) \setminus VJN_p(\mathbb{R}^n)$ and a function $\hat{g}_0 \in CJN_p(\mathbb{R}^n) \setminus L^p(\mathbb{R}^n) / \mathbb{R}$.
\end{corollary}

In conclusion we know that
\begin{equation*}
L^p / \mathbb{R}
\subsetneq CJN_p
\subseteq VJN_p
\subsetneq JN_p
\subsetneq L^{p,\infty} / \mathbb{R}.
\end{equation*}


So far an example function that would be in $VJN_p \setminus CJN_p$ has not been constructed.
It remains an open question whether these two spaces coincide or not.
It is also not clear whether the condition (\ref{auki}) is necessary for Lemma \ref{tokalisäys}. It would be interesting to find a $JN_p$ function that doesn't satisfy (\ref{auki}).
\\

\textbf{Acknowledgements.} I would like to thank my supervisor Riikka Korte
and my colleague Kim Myyryläinen
for many fruitful discussions on the space $JN_p$.
I would also like to thank the referee for their useful comments and suggestions especially concerning the readability of the paper.

This version of the article has been accepted for publication, after peer review, but is not the Version of Record and does not reflect post-acceptance improvements, or any corrections. The Version of Record is available online at: http://dx.doi.org/10.1007/s00209-022-03100-w.

\end{document}